\documentclass[10pt]{article}

\usepackage{srcltx}
\usepackage{hyperref}
\usepackage{tikz}
\usepackage{color}
\usepackage{amsmath,amssymb,amsthm}

\newtheorem{lemma}{Lemma}[section]
\newtheorem{theorem}[lemma]{Theorem}
\newtheorem{corollary}[lemma]{Corollary}
\newtheorem{proposition}[lemma]{Proposition}

\theoremstyle{definition}
\newtheorem{example}[lemma]{Example}
\newtheorem{definition}[lemma]{Definition}
\newtheorem{remark}[lemma]{Remark}

\numberwithin{equation}{section}

\makeatletter
\renewcommand{\p@enumii}{}
\makeatother

\newcommand\wtsmash[1]{{\smash{\widetilde{#1}}\rlap{$\phantom{#1}$}}}

\newcommand{\DD}[1]{\mathbin{\frac{\rm d }{{\rm d }#1}}}

\newcommand{\ud}{\,{\mathrm d}}

\newcommand{\Div}{{\rm div\,}}
\newcommand{\Grad}{{\rm grad\,}}

\newcommand\R{{\mathbb R}}
\newcommand\C{{\mathbb C}}

\newcommand\K{{\mathbb K}}

\newcommand\rplus{{\R_{+}}}
\newcommand\cplus{{\C_{+}}}

\newcommand\zero{\set{0}}

\newcommand{\Hscr}{\mathcal H}

\newcommand{\Uscr}{\mathcal U}

\newcommand{\AB}{{A\& B}}
\newcommand{\CD}{{C\& D}}

\newcommand{\SysNode}{\bbm{\AB \cr \CD}}
\newcommand{\SmallSysNode}{\sbm{\AB \cr \CD}}

\newcommand{\dom}[1]{\mathrm{dom}\left(#1\right)}
\newcommand{\range}[1]{\mathrm{ran}\left(#1\right)}

\newcommand{\Ker}[1]{\ker\left(#1\right)}

\newcommand{\res}[1]{\rho\left(#1\right)}
\newcommand{\re}[0]{\mathrm{Re}\,}

\newcommand{\Ipdp}[2]{\left\langle #1 , #2 \right\rangle}

\newcommand{\set}[1]{\left\lbrace #1 \right\rbrace}

\newcommand{\bigmid}{\bigm\vert}

\newcommand{\biggmid}{\biggm\vert}

\newcommand{\bi}{\begin{itemize}}
\newcommand{\ei}{\end{itemize}}
\newcommand{\be}{\begin{enumerate}}
\newcommand{\ee}{\end{enumerate}}

\newcommand{\Dfrak}{\mathfrak D}

\newcommand{\sbm}[1]{\left[\begin{smallmatrix}#1\end{smallmatrix}\right]}

\newcommand{\bbm}[1]{\begin{bmatrix}#1\end{bmatrix}}

\makeatletter

\def\etv{& \hskip-.3em\vrule\hskip-.3em &} 
\def\smalletv{&\vrule&} 
\def\smallcrh{\vrule height0pt depth2\ex@ width0pt
\cr\noalign{\hrule}
\vrule height6.5\ex@ depth0pt width0pt}
\newbox\smallstrutbox
\setbox\smallstrutbox=\hbox{\vrule height6pt depth1.5pt width\z@}
\def\smallstrut{\relax\ifmmode\copy\smallstrutbox\else\unhcopy\smallstrutbox\fi}

\newenvironment{sysmatrix}{
\let\|=\etv
\hskip \matrixcolsep
\begin{matrix}}
{\end{matrix}
\hskip \matrixcolsep
}       

\newenvironment{smallsysmatrix}{\null\,\vcenter\bgroup
\let\|=\smalletv

\def\\{\smallstrut\math@cr}
\restore@math@cr\default@tag
\baselineskip\z@skip \lineskip\z@skip \lineskiplimit\lineskip
\ialign\bgroup\hfil$\m@th\scriptstyle##$\hfil&&\thickspace\hfil
$\m@th\scriptstyle##$\hfil\crcr
\crcr\noalign{\vskip -.3\ex@}%
}{\crcr\noalign{\vskip -.2\ex@}%
\crcr\egroup\egroup\,%
}

\makeatother

\begin{document}

\title{Feedback theory extended for proving generation of contraction semigroups}

\author{Mikael Kurula\footnote{{Corresponding author, {\tt mkurula@abo.fi}.}}\and Hans Zwart}

\maketitle

\begin{abstract} Recently, the following novel method for proving the
  existence of solutions for certain linear time-invariant PDEs was
  introduced: The operator associated to a given PDE is represented
  by a (larger) operator with an internal loop. If the
  larger operator (without the internal loop) generates a contraction
  semigroup, the internal loop is accretive, and some non-restrictive technical assumptions are
  fulfilled, then the original operator generates a contraction
  semigroup as well. Beginning with the undamped wave equation,
    this general idea can be applied to show that the heat
    equation and wave equations with damping are well-posed. 
In the present paper we show how this approach can benefit 
from feedback techniques and recent developments in well-posed systems theory, at the
same time generalising the previously known results. Among others, we
show how well-posedness of degenerate parabolic equations can be proved.
\end{abstract}



\section{Introduction}

It is now a very standard technique to use semigroup
  theory for showing existence and uniqueness of (linear) partial
  differential equations (PDEs). The general results available in
  semigroup theory enable us to conclude existence of solutions for
  many PDEs once this has been proved for one PDE. For instance, if
  the operator $A$ associated to a given PDE generates a
  $C_0$-semigroup, then we immediately have that for every bounded
  $Q$, also $A+Q$ generates a $C_0$-semigroup. Hence the PDE
  associated to $A+Q$ has a unique solution given an initial
  condition. Even hyperbolic and parabolic PDEs are linked in the
  semigroup setting, since $A^2$ generates an (analytic) semigroup
  whenever $A$ generates a $C_0$-group, see \cite[pp.\
  106--107]{EnNa00}. For contraction semigroups on Hilbert spaces this
  latter result was complemented in \cite{ParaHypConf}.

In \cite{ParaHypConf} it is shown that the existence of solutions of the heat equation,
\begin{equation}\label{eq:heatND}
\left\{
     \begin{aligned}
         \frac{\partial x}{\partial t}(\xi,t) &= \Div\big(\alpha(\xi)\,\Grad x(\xi,t)\big),
 	  \quad \xi\in\Omega,~t\geq0,\\
 	x(\xi,0) &= x_0(\xi),\quad \xi\in\Omega,\\   
 	x(\xi,t) &= 0,\quad \xi\in\partial\Omega,~t\geq0,
     \end{aligned}
   \right.
\end{equation}
can be directly linked to the existence of solutions of the undamped
wave equation
\begin{equation}\label{eq:waveintro}
  \left\{
    \begin{aligned}
   \frac{\partial}{\partial t} \bbm{x(\xi, t)\\\ e(\xi, t)} =& \bbm{0&\mathrm{div}\\\mathrm{grad}&0} \bbm{x(\xi,t)\\e(\xi,t)},
     \quad \xi\in\Omega, ~t\geq0,\\
      \bbm{x(\xi,0)\\e(\xi,0)}&=\bbm{x_0(\xi)\\e_0(\xi)},\quad \xi\in\Omega,\\
    	 x(\xi,t) &= 0,\quad \xi\in\partial\Omega,~t\geq0.
  \end{aligned}\right.
\end{equation}
Here $\Omega\subset {\mathbb R}^n$ is a bounded Lipschitz domain with boundary $\partial\Omega$, $\mathrm{div}$ is the divergence operator $\Div w=\partial w_1/\partial\xi_1+\ldots+\partial w_n/\partial\xi_n$, $\mathrm{grad}$ is the gradient operator $\Grad x=(\partial x/\partial\xi_1,\ldots,\partial x/\partial\xi_n)^\top$, and $\alpha(\xi)$ is the (strictly positive) thermal diffusivity at the point $\xi\in\Omega$. 

The key to this link (more details below) is the next theorem which is taken from \cite[Thm 2.6]{ParaHypConf}. In the theorem, we assume that two (in general unbounded) operators are given: $A_1:\sbm{X_1\\X_2}\supset\dom{A_1}\to X_1$ and $A_{21}:X_1\supset\dom{A_{21}}\to X_2$. Then we define an operator $A_{ext}$ as
\begin{equation}\label{eq:AextIntro2}
\begin{aligned}
	A_{ext} &:= \bbm{A_1\\\bbm{A_{21}&0}}, \\ 
	\dom{A_{ext}}  &:=
		\set{\bbm{x\\e}\in\dom{A_1}\bigmid x\in\dom{A_{21}}}.
\end{aligned}
\end{equation}
\begin{theorem}\label{thm:parahypconf}
Assume that $A_{ext}$ in \eqref{eq:AextIntro2} generates a contraction semigroup on the pair $\sbm{X_1\\X_2}$ of Hilbert spaces and that $S$ is a bounded operator on $X_2$ that satisfies $\re\Ipdp{Sx}{x}\geq\delta\|x\|^2$ for some $\delta>0$ and all $x\in X_2$. 

Then the operator $A_S$ defined using $A_{ext}$ and $S$ as
\begin{equation}\label{eq:ASintro}
\begin{aligned}
  A_Sx &:= A_1\bbm{x\\SA_{21}x},\\
  \dom{A_S} &:= \left\{x\in\dom{A_{21}}\bigmid \bbm{x\\SA_{21}x}\in\dom{A_1}\right\}
\end{aligned}
\end{equation}
generates a contraction semigroup on $X_1$.
\end{theorem}

In order to show how this semigroup-theoretic result links the PDEs
(\ref{eq:heatND}) and (\ref{eq:waveintro}), we have to identify the
spaces and operators of Theorem \ref{thm:parahypconf}. 
As Hilbert spaces $X_1$ and $X_2$ we choose $L^2(\Omega)$ and
$L^2(\Omega)^n$, respectively. The operator $A_{ext}$ is given by
\begin{equation}\label{eq:AextIntro}
  A_{ext} = \bbm{ A_1 \\ \bbm{A_{21}& 0}} := \bbm{0&\mathrm{div}\\\mathrm{grad}&0},
    \quad \dom{A_{ext}} = \bbm{H^1_0(\Omega)\\H^{\mathrm{div}}(\Omega)},
\end{equation}
where $H^1(\Omega)$ is the standard Sobolev space of functions that
together with all their first-order partial derivatives lie in
$L^2(\Omega)$, $H^1_0(\Omega)$ is the 
subspace of functions in $H^1(\Omega)$ that vanish on the boundary $\partial\Omega$ of $\Omega$, and
$$
  H^{\mathrm{div}}(\Omega) := \set{w\in L^2(\Omega)^n\bigmid \Div w\in L^2(\Omega)}.
$$
It is clear that (\ref{eq:waveintro}) is associated to the operator
$A_{ext}$. Since $A_{ext}$ is skew-adjoint on $\sbm{X_1\\X_2}$, see
e.g.\ \cite{parahyp_report}, it generates a contraction
semigroup. Choosing $S$ to be the multiplication operator $(Sf)(\xi)=
\alpha(\xi) f(\xi)$, $f \in X_2$, $\xi\in\Omega$, it is straightforward to see that
$A_S$ in (\ref{eq:ASintro}) is the operator associated to
(\ref{eq:heatND}). Hence if the thermal diffusivity $\alpha$
satisfies the (physically natural) condition $0< m I \leq \alpha(\xi)
\leq M I$, $\xi \in \Omega$ with $m$ and $M$ independent of $\xi$,
then we can use Theorem \ref{thm:parahypconf} to link the two PDEs.  

Theorem \ref{thm:parahypconf} was proved as \cite[Thm
2.6]{ParaHypConf} using a perturbation argument, and the result and
its proof are also included in \cite{ParaHyp}. In the present article
we give a new proof method which also allows us to generalize this
theorem.  In order to formulate our result, we have to introduce some notation
and terminology; the precise definitions are given later in
the paper. 

As in Theorem \ref{thm:parahypconf}, $A_{ext}$ is assumed to generate a contraction
semigroup on $\sbm{X_1\\X_2}$. However, we do not assume that the lower
right corner is zero. This influences the definition of $A_S$ which now becomes $A_Sx:=z$
where $\sbm{z\\f}=A_{ext}\sbm{x\\Sf}$ for some $f$; see
Definition \ref{def:AS}. The \emph{external Cayley system transform} of
$A_{ext}$ is the mapping from $\left[\begin{smallmatrix} x \\
  u \end{smallmatrix}\right]$ to $\left[\begin{smallmatrix} z \\
  y \end{smallmatrix}\right]$, where $\sbm{z\\f}:=A_{ext}\sbm{x\\e}$ and
$u:=\frac{e-f}{\sqrt2}$, $y:=\frac{e+f}{\sqrt2}$. A \emph{system node} is the natural generalization to
infinite-dimensional systems of the matrix $\sbm{A&B\\C&D}$ in the
continuous-time finite-dimensional system $\sbm{\dot
  x(t)\\y(t)}=\sbm{A&B\\C&D}\sbm{x(t)\\u(t)}$; see Definition
\ref{def:sysnode}. A system node is \emph{scattering passive} if and only if the
following energy inequality holds:
$$
  2\re\Ipdp{z}{x}_X\leq \|u\|_U^2-\|y\|_Y^2,\quad \text{with}~\bbm{z\\y}=\bbm{\AB\\\CD}\bbm{x\\u}.
$$
\begin{theorem}\label{thm:KadmConseq}
Let $A_{ext}$ generate a contraction semigroup on the pair
$\sbm{X_1\\X_2}$ of Hilbert spaces and let $-S$ generate a contraction
semigroup on $X_2$. Then the following claims are true:
\begin{enumerate}
\item
  The external Cayley system transform $\SmallSysNode$ of $A_{ext}$ is
  a scattering-passive system node.
\item
  If the Cayley transform $K=(S-I)(S+I)^{-1}$ of $S$ is an admissible static output feedback
  operator for $\SmallSysNode$, then the relation $A_S$ defined via
$$
  A_Sx:=z,\quad \bbm{z\\f}=A_{ext}\bbm{x\\Sf}~ \text{for some}~ f\in X_2,
$$
is the (single-valued) generator of a contraction semigroup on $X_1$.
\item 
   If $S$ is bounded and $\re\Ipdp{Sx}{x}\geq\delta\|x\|^2$ for
  some $\delta>0$ and all $x\in X_2$, then $K$ is admissible.
\end{enumerate}
\end{theorem}

The converse of assertion two in Theorem \ref{thm:KadmConseq} is false; the operator $A_S$ may generate a contraction semigroup even though $K$ is not admissible; see Example \ref{ex:mainconvfalse}.

\begin{remark}\label{rem:insights}
Intuitively, assertion 2 of Theorem \ref{thm:KadmConseq} says that if the closed loop system $\SmallSysNode$ with static output feedback $K$ is a meaningful control system, i.e., a system node, then the main operator of this system generates a contraction semigroup. For more details, see Definition \ref{def:feedb} and Proposition \ref{prop:Sec2Concl} below.

Claim 1 of the previous theorem follows from \cite[Theorem 5.2]{StWe12b}. If $S$ in item 2 or 3 is the identity, then $K=0$ and we have no feedback. Moreover, in this case $A_S$ equals the main operator $A$ in item 1; see Theorem \ref{thm:stwe}. It is often convenient to make the canonical choice $S=I$ (which corresponds to $\alpha\equiv I$ in the heat equation in \eqref{eq:heatND}), but in many cases this is not preferable, or even possible. For instance, the wave equation can be transformed into the viscous Schr\"odinger equation \cite{ErZu09} by choosing $S=iI+\varepsilon$. Although the solutions of the heat and Schr\"odinger equations have completely different properties, the existence of solutions can in both cases be proved by applying Theorem \ref{thm:KadmConseq} to the wave equation (\ref{eq:waveintro}). The examples in Sections \ref{sec:plates} and \ref{sec:degenerate} have $S\neq I$. In this paper we focus on the case described in item 3; hence in all examples the operators $S$ are bounded. 

In the case of item 3, the closed loop system $\SmallSysNode$ with static output feedback $K$ is even well-posed in the sense of Definition \ref{def:isopass}. This will be pointed out later in the discussion.
\end{remark}

We do not expect that
  Theorem \ref{thm:KadmConseq} can yield existence of solutions for a PDE for which no direct
  solution method exists. Rather, our point is that feedback theory can
  quickly solve the problem of existence of solutions, once the problem is solved
  for a simpler PDE; see Section \ref{sec:plates}. Furthermore, it follows from our method
  that not only homogeneous PDEs are well-posed, but also the
  well-posedness of some inhomogeneous PDEs is obtained;
 see Example \ref{ex:wavesysnode}. In a companion paper
 \cite{parahyp_divgrad} we have shown how to easily characterise the
 boundary conditions which give rise to a contraction semigroup for
 many hyperbolic PDEs, especially those similar to the wave
 equation. Controllability and observability of the heat equation have
 previously been successfully studied using the corresponding
 properties of the associated wave equation in \cite{FaRu71}; see
 \cite{Mil06,ErZu11,ParaHyp} for more recent developments in this
 area.

The full abstract setting of the paper is described in detail in
Section \ref{sec:background}, together with a minimal background on
continuous-time infinite-dimensional systems theory. In Section 3, we
transform the maximal dissipative operator $A_{ext}$ into a
scattering-passive system node $\SmallSysNode$, using a recent result
on the external Cayley system transformation by Staffans and Weiss; see \cite[Thm 4.6]{StWe12a}. Then we proceed to represent $A_S$ in terms of $\SmallSysNode$. The main contribution of the paper is Section \ref{sec:scattfeedback}, where we prove Theorem \ref{thm:KadmConseq} using feedback techniques. We apply the results of Theorem \ref{thm:KadmConseq} in Section \ref{sec:plates}, where two examples of damped wave equations are provided, one with viscous damping and one with structural damping.
We end the paper with an application of Theorem \ref{thm:KadmConseq} to degenerate parabolic PDEs, in Section \ref{sec:degenerate}.

Theorem \ref{thm:parahypconf} was generalized to the Banach-space setting by Schwenninger in \cite{ParaBanach}. The work \cite{TuWe03,WeTu03,TuWeBook} of Tucsnak and Weiss on ``conservative systems from thin air'', and that of Staffans and Weiss \cite{StWe12a,StWe12b} on Maxwell's equations, are also very closely related to the present paper. However, it is not straightforward to translate the results from one setting to the other, and neither approach seems to be a special case of the other one. In the present paper we make extensive use of well-posed systems theory \cite{StafBook}, and useful connections can also be made to linear port-Hamiltonian systems \cite{JaZwBook,JaviThesis, KZSB10}. Finally, it should be mentioned that Desoer and Vidyasagar used similar methods with finite-dimensional, but non-linear, systems in \cite[Sect.\ VI.5]{DeViBook}.

\section{The abstract setting}\label{sec:background}

The operator $A$ on a Hilbert space $X$ is \emph{dissipative} if $\re\Ipdp{A x}{x}\leq 0$ for all $x\in\dom{A}$, and we say that $A$ is \emph{maximal dissipative} if $A$ has no proper extension which is still a dissipative operator on $X$. The operator $S$ is \emph{(maximal) accretive} if $-S$ is (maximal) dissipative. The following definition generalizes \eqref{eq:ASintro}; see Figure \ref{fig:AS} for an illustration:

\begin{definition}\label{def:AS}
Let $X_1$ and $X_2$ be Hilbert spaces, let $A_{ext}=\sbm{A_1\\A_2}:\sbm{X_1\\X_2}\supset\dom{A_{ext}}\to\sbm{X_1\\X_2}$ be a closed and maximal dissipative linear operator, and let $S$ be a closed and maximal accretive linear operator on $X_2$. 

The in general unbounded mapping $A_S$ from $\dom{A_S}\subset X_1$ into $X_1$ defined by
\begin{equation}\label{eq:ASDef}
\begin{aligned}
  \dom{A_S} &:= \bigg\{x\in X_1\bigmid \exists f\in\dom S,e\in X_2:\\ &\qquad\qquad
    ~\bbm{x\\e}\in\dom{A_{ext}},~ f=A_2\bbm{x\\e},~ e=Sf\bigg\},\\
  A_S x &:= z,\quad \bbm{z\\f}=A_{ext}\bbm{x\\e},~e=Sf,
\end{aligned}
\end{equation}
is called the mapping \emph{$A_{ext}$ with internal loop through $S$}.
\end{definition}

\begin{figure}[!h]
\begin{center}

\ifx\du\undefined
  \newlength{\du}
\fi
\setlength{\du}{5\unitlength}
\begin{tikzpicture}
\pgftransformxscale{1.000000}
\pgftransformyscale{-1.000000}
\definecolor{dialinecolor}{rgb}{0.000000, 0.000000, 0.000000}
\pgfsetstrokecolor{dialinecolor}
\definecolor{dialinecolor}{rgb}{1.000000, 1.000000, 1.000000}
\pgfsetfillcolor{dialinecolor}
\pgfsetlinewidth{0.100000\du}
\pgfsetdash{{1.000000\du}{1.000000\du}}{0\du}
\pgfsetdash{{1.000000\du}{1.000000\du}}{0\du}
\pgfsetroundjoin
{\pgfsetcornersarced{\pgfpoint{0.700000\du}{0.700000\du}}\definecolor{dialinecolor}{rgb}{0.000000, 0.000000, 0.000000}
\pgfsetstrokecolor{dialinecolor}
\draw (-5.000000\du,-1.000000\du)--(-5.000000\du,20.000000\du)--(27.000000\du,20.000000\du)--(27.000000\du,-1.000000\du)--cycle;
}\pgfsetlinewidth{0.100000\du}
\pgfsetdash{}{0pt}
\pgfsetdash{}{0pt}
\pgfsetroundjoin
{\pgfsetcornersarced{\pgfpoint{0.300000\du}{0.300000\du}}\definecolor{dialinecolor}{rgb}{0.000000, 0.000000, 0.000000}
\pgfsetstrokecolor{dialinecolor}
\draw (6.000000\du,4.000000\du)--(6.000000\du,10.000000\du)--(16.000000\du,10.000000\du)--(16.000000\du,4.000000\du)--cycle;
}
\definecolor{dialinecolor}{rgb}{0.000000, 0.000000, 0.000000}
\pgfsetstrokecolor{dialinecolor}
\node[anchor=west] at (22.000000\du,2.000000\du){$A_S$};
\definecolor{dialinecolor}{rgb}{0.000000, 0.000000, 0.000000}
\pgfsetstrokecolor{dialinecolor}
\node[anchor=west] at (8.500000\du,7.000000\du){$A_{ext}$};
\pgfsetlinewidth{0.100000\du}
\pgfsetdash{}{0pt}
\pgfsetdash{}{0pt}
\pgfsetroundjoin
{\pgfsetcornersarced{\pgfpoint{0.300000\du}{0.300000\du}}\definecolor{dialinecolor}{rgb}{0.000000, 0.000000, 0.000000}
\pgfsetstrokecolor{dialinecolor}
\draw (8.000000\du,13.000000\du)--(8.000000\du,17.000000\du)--(14.000000\du,17.000000\du)--(14.000000\du,13.000000\du)--cycle;
}
\definecolor{dialinecolor}{rgb}{0.000000, 0.000000, 0.000000}
\pgfsetstrokecolor{dialinecolor}
\node[anchor=west] at (9.450000\du,15.000000\du){$S$};
\pgfsetlinewidth{0.100000\du}
\pgfsetdash{}{0pt}
\pgfsetdash{}{0pt}
\pgfsetmiterjoin
\pgfsetbuttcap
{
\definecolor{dialinecolor}{rgb}{0.000000, 0.000000, 0.000000}
\pgfsetfillcolor{dialinecolor}
\pgfsetarrowsend{stealth}
{\pgfsetcornersarced{\pgfpoint{0.000000\du}{0.000000\du}}\definecolor{dialinecolor}{rgb}{0.000000, 0.000000, 0.000000}
\pgfsetstrokecolor{dialinecolor}
\draw (16.000000\du,8.000000\du)--(22.000000\du,8.000000\du)--(22.000000\du,15.000000\du)--(14.000000\du,15.000000\du);
}}
\pgfsetlinewidth{0.100000\du}
\pgfsetdash{}{0pt}
\pgfsetdash{}{0pt}
\pgfsetmiterjoin
\pgfsetbuttcap
{
\definecolor{dialinecolor}{rgb}{0.000000, 0.000000, 0.000000}
\pgfsetfillcolor{dialinecolor}
\pgfsetarrowsend{stealth}
{\pgfsetcornersarced{\pgfpoint{0.000000\du}{0.000000\du}}\definecolor{dialinecolor}{rgb}{0.000000, 0.000000, 0.000000}
\pgfsetstrokecolor{dialinecolor}
\draw (8.000000\du,15.000000\du)--(0.000000\du,15.000000\du)--(0.000000\du,8.000000\du)--(6.000000\du,8.000000\du);
}}
\definecolor{dialinecolor}{rgb}{0.000000, 0.000000, 0.000000}
\pgfsetstrokecolor{dialinecolor}
\node[anchor=west] at (8.000000\du,7.000000\du){};
\pgfsetlinewidth{0.100000\du}
\pgfsetdash{}{0pt}
\pgfsetdash{}{0pt}
\pgfsetbuttcap
{
\definecolor{dialinecolor}{rgb}{0.000000, 0.000000, 0.000000}
\pgfsetfillcolor{dialinecolor}
\pgfsetarrowsend{stealth}
\definecolor{dialinecolor}{rgb}{0.000000, 0.000000, 0.000000}
\pgfsetstrokecolor{dialinecolor}
\draw (-10.000000\du,6.000000\du)--(6.000000\du,6.000000\du);
}
\pgfsetlinewidth{0.100000\du}
\pgfsetdash{}{0pt}
\pgfsetdash{}{0pt}
\pgfsetbuttcap
{
\definecolor{dialinecolor}{rgb}{0.000000, 0.000000, 0.000000}
\pgfsetfillcolor{dialinecolor}
\pgfsetarrowsend{stealth}
\definecolor{dialinecolor}{rgb}{0.000000, 0.000000, 0.000000}
\pgfsetstrokecolor{dialinecolor}
\draw (16.000000\du,6.000000\du)--(34.000000\du,6.000000\du);
}
\definecolor{dialinecolor}{rgb}{0.000000, 0.000000, 0.000000}
\pgfsetstrokecolor{dialinecolor}
\node[anchor=west] at (8.000000\du,7.000000\du){};
\definecolor{dialinecolor}{rgb}{0.000000, 0.000000, 0.000000}
\pgfsetstrokecolor{dialinecolor}
\node[anchor=west] at (9.000000\du,7.000000\du){};
\definecolor{dialinecolor}{rgb}{0.000000, 0.000000, 0.000000}
\pgfsetstrokecolor{dialinecolor}
\node[anchor=west] at (28.000000\du,4.000000\du){$z=A_Sx$};
\definecolor{dialinecolor}{rgb}{0.000000, 0.000000, 0.000000}
\pgfsetstrokecolor{dialinecolor}
\node[anchor=west] at (15.000000\du,17.000000\du){$f$};
\definecolor{dialinecolor}{rgb}{0.000000, 0.000000, 0.000000}
\pgfsetstrokecolor{dialinecolor}
\node[anchor=west] at (-10.000000\du,4.000000\du){$x$};
\definecolor{dialinecolor}{rgb}{0.000000, 0.000000, 0.000000}
\pgfsetstrokecolor{dialinecolor}
\node[anchor=west] at (4.000000\du,17.000000\du){$e$};
\definecolor{dialinecolor}{rgb}{0.000000, 0.000000, 0.000000}
\pgfsetstrokecolor{dialinecolor}
\node[anchor=west] at (27.000000\du,4.000000\du){};
\definecolor{dialinecolor}{rgb}{0.000000, 0.000000, 0.000000}
\pgfsetstrokecolor{dialinecolor}
\node[anchor=west] at (18.150000\du,16.850000\du){};
\end{tikzpicture}

\caption{Representing $A_S$ using $A_{ext}$ and $S$.}
\label{fig:AS}
\end{center}
\end{figure}
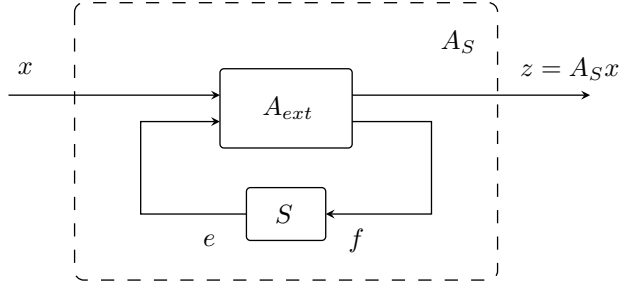

If $A_{ext}$ is of the form $\sbm{A_1\\\sbm{A_{21}&0}}$ then
\eqref{eq:ASDef} reduces to \eqref{eq:ASintro}. Moreover, it is
straightforward to verify that $A_S$ is linear, but $A_S$ can in
general be multi-valued, even when both $A_{ext}$ and $S$ are
single-valued. For example, take $X_1=X_2=\C$,
$A_{ext}=\sbm{0&i\\i&-i}$, and $S=i$. Then $\dom{A_S}=\zero$ and the
multi-valued part of $A_S$ is $\C$. Fortunately, in the combinations
of $A_{ext}$ and $S$ that we consider in the present paper $A_S$ is
always single valued. 

\begin{remark}
Figure \ref{fig:AS} is strongly reminiscent of feedback, but we want to emphasize that we are at this point \emph{not working with standard feedback}. In the ODE \eqref{eq:waveintro} associated to $A_{ext}$, both variables $x$ and $e$ are state variables of a system that has no inputs or outputs. On the other hand, if we want to interpret Figure \ref{fig:AS} as feedback, then $e$ would have the interpretation of input signal, $x$ would be the state variable, $z=\dot x$ the (time) derivative of the state, and $f$ would be the output signal. Sometimes, but certainly not always, it is possible to obtain useful results by making such a reinterpretation of the variables. For instance, if we in \eqref{eq:waveintro} replace $\dot e$ by an arbitrary variable $f$, and thus drop the assumption that $f=\dot e$, then we no longer have a meaningful system in the sense that the resulting mapping from $\sbm{x\\e}$ to $\sbm{z\\f}$ is not a system node. See Definition \ref{def:sysnode} and Example \ref{ex:wavesysnode} below for more details.
\end{remark}

The operator $A_S$ is always dissipative if $A_{ext}$ is dissipative and $S$ is accretive. Indeed, due to \eqref{eq:ASDef}, we can for all $x\in\dom{A_S}$ find $z\in X_1$ $f,e\in X_2$ such that $\sbm{z\\f}=A_{ext}\sbm{x\\e}$ and $e=Sf$. Then $z=A_Sx$ and it follows that
\begin{equation}\label{eq:ASdiss}
\begin{aligned}
  \re\Ipdp{A_Sx}{x} &= \re\Ipdp{z}{x} = \re \Ipdp{\bbm{z\\f}}{\bbm{x\\e}} - \re \Ipdp{f}{e} \\
  &= \re \Ipdp{A_{ext}\bbm{x\\e}}{\bbm{x\\e}} -\re \Ipdp{f}{Sf}\leq 0.
\end{aligned}
\end{equation}
According to the following famous theorem, $A_S$ generates a contraction semigroup
if and only if $A_S$ is closed and \emph{maximal} dissipative:

\begin{theorem}[Lumer-Phillips]\label{thm:lumphi}
For a linear operator $A$ on a Hilbert space $X$, the following conditions are equivalent:
\begin{enumerate}
\item $A$ generates a contraction semigroup on $X$.
\item $A$ is closed and maximal dissipative, i.e., dissipative with no dissipative proper extension.
\item $A$ is densely defined, closed, and dissipative, and $A^*$ is also dissipative.
\item $A$ is dissipative and there exists at least one $\alpha\in\cplus=\set{\lambda\in\C\mid \re \lambda>0}$ such that $\range{\alpha I-A}=X$.
\item $A$ is dissipative and $\alpha I-A$ has a bounded inverse on $X$ for every $\alpha\in\cplus$. 
\end{enumerate}
\end{theorem}

The standard definition of a contraction semigroup and additional background can be found in most books on semigroup theory. Here we assume that the reader is familiar with this theory, and we refer to Chapter 3 of \cite{StafBook} for more details. For a proof of Theorem \ref{thm:lumphi}, see in particular \cite[Thms 3.4.8 and 3.4.9]{StafBook}, noting that $\alpha I-A$ is always injective when $\alpha\in\cplus$ and $A$ is dissipative. The importance of the assumption that $A$ is closed in item two of the Lumer-Phillips Theorem was investigated in \cite[\S I.1.1]{PhiMaxdiss}.

Let $X$ be a Hilbert space and $A$ a linear operator defined on some subset of $X$. Defining the \emph{resolvent set} of $A$ to be the set $\res{A}$ of all $\lambda\in\C$ for which $\lambda I-A$ is both injective and surjective, we can state assertion 4 of Theorem \ref{thm:lumphi} equivalently as ``$A$ is dissipative and $\cplus\cap\res{A}\neq\emptyset$''. Similarly, assertion 5 is equivalent to ``$A$ is dissipative and $\cplus\subset\res{A}$'', due to the Closed Graph Theorem.

Next we introduce the concept of a \emph{system node}. It is helpful to think about a system node $\SmallSysNode$ as a generalization to infinite dimensions of the matrix $\sbm{A&B\\C&D}$ in the standard finite-dimensional linear system with input signal $u(\cdot)$, state trajectory $x(\cdot)$, and output signal $y(\cdot)$:
\begin{equation}\label{eq:finitedim}
  \bbm{\dot x(t)\\y(t)} = \bbm{A&B\\C&D} \bbm{x(t)\\u(t)},\quad t\geq0,\quad x(0)=x_0.
\end{equation}
The associated semigroup is the mapping $t\mapsto e^{At}$, $t\geq0$, which for zero input $u(t)=0$, $t\geq0$, sends the initial state $x_0$ into the state $x(t)$ at time $t\geq0$.

The following definition of a system node is slightly different from the standard definition \cite[Def.\ 4.7.2]{StafBook} that uses rigged Hilbert spaces, but the definitions are seen to be equivalent by combining \cite[Lem.\ 4.7.7]{StafBook} with the fact that every generator of a $C_0$-semigroup has a non-empty resolvent set; see \cite[Thm 3.2.9]{StafBook}.

\begin{definition}\label{def:sysnode}
By a \emph{system node} with \emph{input space} $U$, \emph{state space} $X$, and \emph{output space} $Y$, all Hilbert spaces, we mean an in general unbounded linear operator
$$\bbm{\AB\\\CD}:\bbm{ X\\ U}\supset\dom {\bbm{\AB\\\CD}}\to \bbm{ X\\ Y}$$
with the following properties:
\begin{enumerate}
\item The operator $\SmallSysNode$ is closed.
\item The operator $\AB$ is closed, where $\AB$ is the projection of $\SmallSysNode$ onto $\sbm{X\\\zero}$ along $\sbm{\zero\\Y}$.
\item the \emph{main operator} $A \colon \dom A \to  X$, which is defined by 
\begin{equation}\label{eq:Adef}
	Ax = \AB\bbm{x\\ 0},\quad \dom A = \set{x \in  X \bigmid \bbm{x \\ 0} \in \dom{\SmallSysNode}},
\end{equation}
is the generator of a $C_0$-semigroup on $ X$.
\item The domain of $\SmallSysNode$ satisfies the condition
$$
  \forall u\in U ~ \exists x\in X:\quad \bbm{x\\u} \in \dom{\SmallSysNode}.
$$
\end{enumerate}

By a \emph{classical trajectory} of the system node $\SmallSysNode$ we mean a triple $(u,x,y)$ where $u\in C(\rplus; U)$, $x\in C^1(\rplus; X)$, $y\in C(\rplus; Y)$, $\sbm{x(t)\\u(t)}\in\dom{\SmallSysNode}$ for all $t\geq0$, and
\begin{equation}\label{eq:classtraj}
  \bbm{\dot x(t)\\y(t)}=\bbm{\AB\\\CD} \bbm{x(t)\\u(t)},\quad t\geq0,
\end{equation}
using the derivative from the right at $0$.
\end{definition}

When we in the sequel use the notation $\AB$, we
mean that the operators $A$ and $B$ can in general no longer be separated from
each other (without extending the co-domain and the domain).

Let $\pi_{[0,T]}$ denote the linear operator which first restricts a function to the interval $[0,T]$ and then extends the restricted function by zero on $\R\setminus[0,T]$, and introduce the Sobolev space
$$
  H^1_0(\rplus;U):=\set{u\in L^2(\rplus;U)\bigmid \frac{\mathrm{d}u}{\mathrm{d}\xi}\in L^2(\rplus;U), ~u(0)=0}.
$$
Let $\SmallSysNode$ be a system node. Then there for every $u\in H^1_0(\rplus;U)$ exist $x\in C^1(\rplus;X)$ and $y\in C(\rplus;Y)$, such that $(u,x,y)$ is a classical trajectory of $\SmallSysNode$ with $x(0)=0$; see \cite[Lemma 4.7.8]{StafBook}. Thus, for every $u_0\in \pi_{[0,T]}H^1_0(\rplus;U)$ there exists a classical trajectory $(u,x,y)$ of $\SmallSysNode$ with $x(0)=0$ and $\pi_{[0,T]}u=u_0$. It is well-known that $\pi_{[0,T]} H^1_0(\rplus;U)$ is dense in $L^2([0,T];U)$, and therefore for all $T>0$,
\begin{equation}\label{eq:classinputdense}
  \Uscr_0^T := \set{\pi_{[0,T]}u\bigmid (u,x,y)~\text{classical trajectory of}~\SmallSysNode~\wedge~x(0)=0} 
\end{equation}
is a dense subspace of $L^2([0,T];U)$.

\begin{definition}\label{def:isopass}
A system node $\SmallSysNode$ is \emph{($L^2$-)well-posed} if there for every $T\geq0$ exists a corresponding constant $M_T\geq0$ such that all classical trajectories $(u,x,y)$ of $\SmallSysNode$ satisfy
\begin{equation}\label{eq:scattpass}
  \|x(T)\|_ X^2 + \int_0^T\|y(t)\|_ Y^2\ud t \leq M_T\left(\|x(0)\|_ X^2 + \int_0^T\|u(t)\|_ U^2\ud t\right).
\end{equation}
The system node is \emph{(scattering) passive} if \eqref{eq:scattpass} holds with $M_T=1$ for all $T\geq0$.
\end{definition}

A system node is well-posed (passive) if and only if there exist one $T>0$, such that the inequality in \eqref{eq:scattpass} holds with some $M_T\geq0$ (with $M_T=1$). Often $M_T$ grows with growing $T$ in the non-passive well-posed case. By \cite[Thm 11.1.5]{StafBook}, a system node $\SmallSysNode$ is passive if and only if it for all $\sbm{x\\u}\in\dom{\SmallSysNode}$ holds that
\begin{equation}\label{eq:scattpassdiff}
  2\re\Ipdp{z}{x}_X\leq \|u\|_U^2-\|y\|_Y^2,\quad \text{with}~\bbm{z\\y}=\bbm{\AB\\\CD}\bbm{x\\u}.
\end{equation}

Let $(u,x,y)$ be a classical trajectory with $x(0)=0$ of a well-posed system node $\SmallSysNode$ and fix $T>0$ arbitrarily. The mapping $\Dfrak_0^T$ from $\pi_{[0,T]}u$ into $\pi_{[0,T]}y$ is a linear operator defined on $\dom{\Dfrak_0^T}=\Uscr_0^T$ with values in $L^2([0,T];Y)$. The domain $\Uscr_0^T$ of $\Dfrak_0^T$ is dense in $L^2([0,T];U)$, and as an operator from $L^2([0,T];U)$ into $L^2([0,T];Y)$, the operator $\Dfrak_0^T$ is bounded by $M_T$ in \eqref{eq:scattpass}.

\begin{definition}\label{def:DzeroT}
Let $\SmallSysNode$ be a well-posed system node and $T>0$ be arbitrary. We call the unique extension of $\Dfrak_0^T$ into a bounded operator $\Dfrak_0^T:L^2([0,T];U)\to L^2([0,T];Y)$ the \emph{$T$-input/output map} of $\SmallSysNode$, and we also denote this extension by $\Dfrak_0^T$.
\end{definition}

For a passive system node, $\Dfrak_0^T$ is a contraction, again by \eqref{eq:scattpass}.

\begin{remark}\label{rem:StaffansD}
Combining Definition 2.2.7 and Theorem 4.6.11 in \cite{StafBook} with our derivation of $\Dfrak_0^T$, we see that our operator $\Dfrak_0^T$ coincides with the operator represented by the same notation in \cite{StafBook}. Indeed, $\Dfrak_0^T$ maps an input signal $u\in L^2([0,T];U)$ into the corresponding output signal $y\in L^2([0,T];Y)$ of a mild trajectory $(u,x,y)$ of $\SmallSysNode$ with $x(0)=0$; see also \cite[Sect.\ 2.1]{StafBook}. 

Compared to our derivation of $\Dfrak_0^T$, Staffans \cite{StafBook} proceeds in the opposite direction. More precisely, he considers an extension $\Dfrak$ of the operator $\Dfrak_0^T$ to the space of functions in $L^2_{loc}(\R;U)$, with support bounded from the left, to be part of the definition of a well-posed system.  Using the operator $\Dfrak$, he defines $\Dfrak_0^T$ by $\Dfrak_0^T:=\pi_{[0,T]} \Dfrak\pi_{[0,T]}$ in Definition 2.2.6, and only later he defines the system node and classical trajectories.
\end{remark}

We end the section with a result that is useful when working on examples. The simple proof, which uses causality and the identity $\Dfrak_0^\tau:=\pi_{[0,\tau]} \Dfrak\pi_{[0,\tau]}$, $\tau>0$, is omitted.

\begin{proposition}\label{prop:DzeroTdecreases1}
For a well-posed system $\SmallSysNode$ with input space $U$ and output space $Y$, the norm of $\Dfrak_0^T$, $T>0$, as an operator from $L^2([0,T];U)$ to $L^2([0,T];Y)$, is a non-decreasing function of $T$. 
\end{proposition}

\section{Representing $A_S$ using a passive system node}\label{sec:cayley}

We provide sufficient conditions for $A_S$ to generate a contraction semigroup by using the following theorem, which is a reformulation of Theorem 5.2 in \cite{StWe12b}. We give a new elementary and self-contained proof, where we show directly that the conditions of Definition \ref{def:sysnode} and the inequality \eqref{eq:scattpassdiff} are satisfied.
\begin{theorem}\label{thm:stwe}
If $A_{ext}=\sbm{A_1\\A_2}$ is a closed and maximal dissipative operator on the pair $\sbm{X_1\\X_2}$ of Hilbert spaces, then the \emph{external Cayley system transform}
\begin{equation}\label{eq:sysnodecayl}
\begin{aligned}
  \bbm{\AB\\\CD} &: \bbm{x\\(e-f)/\sqrt2} \mapsto \bbm{z\\(e+f)/\sqrt2},\quad \bbm{z\\f}=A_{ext}\bbm{x\\e},\\
  \dom{\SmallSysNode} &= \set{\bbm{x\\(e-f)/\sqrt2} \bigmid \bbm{x\\e}\in\dom{A_{ext}},~ f=A_2\bbm{x\\e}},
\end{aligned}
\end{equation}
of $A_{ext}$ is a passive (in particular well posed) system node, with state space $X_1$, and input and output space $X_2$.

The main operator $A$ equals $A_S$ of Definition \ref{def:AS} for $S=I$.
\end{theorem}
\begin{proof}
The following useful equivalence is straightforward to verify:
\begin{equation}\label{eq:caylandinv}
  u=\frac{e-f}{\sqrt2}\quad\text{and}\quad y=\frac{e+f}{\sqrt2}\quad\Longleftrightarrow\quad
  f=\frac{y-u}{\sqrt2}\quad\text{and}\quad e=\frac{y+u}{\sqrt2}.
\end{equation}

In order to prove \eqref{eq:scattpassdiff}, we let $\sbm{x\\u}\in\dom{\SmallSysNode}$ be arbitrary, and we set $\sbm{z\\y}:=\SmallSysNode\sbm{x\\u}$, $e:=(y+u)/\sqrt 2$, and $f:=(y-u)/\sqrt 2$. Then $\sbm{x\\e}\in\dom{A_{ext}}$ and $\sbm{z\\f}=A_{ext}\sbm{x\\e}$ by \eqref{eq:sysnodecayl} and \eqref{eq:caylandinv}, and the dissipativity of $A_{ext}$ yields
\begin{equation}\label{eq:energyineqpf}
\begin{aligned}
  0 &\geq 2\re\Ipdp{\bbm{z\\f}}{\bbm{x\\e}} 
  = 2\re\Ipdp{z}{x}+2\re\Ipdp{\frac{y-u}{\sqrt2}}{\frac{y+u}{\sqrt2}} \\
  &\geq 2\re\Ipdp{z}{x} + \|y\|^2 - \|u\|^2.
\end{aligned}
\end{equation}
We have proved \eqref{eq:scattpassdiff}, and setting $u=0$, we obtain for all $x\in\dom A$ that $z=Ax$ and $2\re \Ipdp{Ax}{x}\leq -\|y\|^2\leq 0$, where $A$ is the main operator of $\SmallSysNode$; see \eqref{eq:Adef}. Hence, $A$ is dissipative.

As $A_{ext}$ is \emph{maximal} dissipative, $1\in\res{A_{ext}}$ by the Lumer-Phillips Theorem \ref{thm:lumphi}, which implies that the operator $I-A_{ext}$ has range $\sbm{X_1\\X_2}$. Therefore, for arbitrary $\wtsmash x\in X_1$ and $\wtsmash u\in X_2$ there exists an $\sbm{x\\e}\in\dom{A_{ext}}$ such that
$$
  \bbm{\wtsmash x\\\sqrt2\, \wtsmash u} = \left(\bbm{I&0\\0&I}-A_{ext}\right)\bbm{x\\e} = \bbm{x-z\\e-f}
      \quad \text{with}\quad \bbm{z\\f}=A_{ext}\bbm{x\\e}.
$$
Comparing this to \eqref{eq:sysnodecayl}, we see that condition 4 of Definition \ref{def:sysnode} is met. Moreover, setting $\wtsmash u=0$, we see that $I-A$ is surjective, and since we already know that $A$ is dissipative, we can conclude that $A$ is maximal dissipative, hence the generator of a contraction semigroup. Thus condition 3 of Definition \ref{def:sysnode} is also met.

Next we prove that $\SmallSysNode$ inherits closedness from $A_{ext}$. Indeed, let $\sbm{x_n\\u_n}\in\dom{\SmallSysNode}$, $\sbm{z_n\\y_n}=\SmallSysNode\sbm{x_n\\u_n}$, $x_n\to x$ and $z_n\to z$ in $X_1$, and $u_n\to u$ and $y_n\to y$ in $X_2$. Then 
$$
  e_n := \frac{y_n+u_n}{\sqrt 2}\to\frac{y+u}{\sqrt2}=:e\quad\text{and}\quad
  f_n := \frac{y_n-u_n}{\sqrt 2}\to \frac{y-u}{\sqrt2}=:f
$$ 
in $X_2$, and moreover $\sbm{x_n\\e_n}\in\dom{A_{ext}}$ with $\sbm{z_n\\f_n}=A_{ext}\sbm{x_n\\e_n}$ due to \eqref{eq:sysnodecayl} and \eqref{eq:caylandinv}. By the closedness of $A_{ext}$, $\sbm{x\\e}\in\dom{A_{ext}}$ and $\sbm{z\\f}=A_{ext}\sbm{x\\e}$, and since $u=(e-f)/\sqrt 2$ and $y=(e+f)/\sqrt 2$ by \eqref{eq:caylandinv}, we obtain from \eqref{eq:sysnodecayl} that $\sbm{x\\u}\in\dom{\SmallSysNode}$ and $\sbm{z\\y}=\SmallSysNode\sbm{x\\u}$. We have proved that $\SmallSysNode$ is closed, and so condition 1 of Definition \ref{def:sysnode} is met.

Finally we need to show that condition 2 of Definition \ref{def:sysnode} is satisfied. Therefore we assume that $\sbm{x_n\\u_n}\in\dom{\SmallSysNode}$, $\sbm{z_n\\y_n}=\SmallSysNode\sbm{x_n\\u_n}$, $x_n\to x$ and $z_n\to z$ in $X_1$, and $u_n\to u$ in $X_2$. Then $x_n$, $z_n$, and $u_n$ are all Cauchy sequences such that $\sbm{z_n-z_m\\y_n-y_m}=\SmallSysNode\sbm{x_n-x_m\\u_n-u_m}$, and combining \eqref{eq:energyineqpf} with the Cauchy-Schwarz inequality, we obtain that
$$
\begin{aligned}
  \|y_n-y_m\|^2 &\leq \|u_n-u_m\|^2 - 2\re\Ipdp{z_n-z_m}{x_n-x_m} \\
  &\leq \|u_n-u_m\|^2 + 2\|z_n-z_m\|\|x_n-x_m\|.
\end{aligned}
$$
This implies that $y_n$ is also a Cauchy sequence in $X_2$. Hence $y_n$ also converges to some $y\in X_2$ and by the closedness of $\SmallSysNode$, we have that $\sbm{x\\u}\in\dom{\AB}$ and $\AB\sbm{x\\u}=z$. We conclude that $\AB$ is closed.

By equation (\ref{eq:Adef}), $x$ is mapped to $z=Ax$ whenever $u=0$. However, $u=0$ corresponds to $e=f$, see (\ref{eq:caylandinv}). Hence $e=If$, and so (\ref{eq:ASDef}) gives $z=A_Ix$.
\end{proof}

In \cite{StafMaxdiss} it was shown that there exist maximal scattering dissipative operators which are not closed, and so the closedness assumption in Theorem \ref{thm:stwe} is essential.

The following alternative representation of the operator $\SmallSysNode$ in \eqref{eq:sysnodecayl} is useful in computations; see also \cite{StWe12a}:

\begin{proposition}\label{prop:caylalt}
Let $A_{ext}=\sbm{A_1\\A_2}$ be a dissipative operator on the pair $\sbm{X_1\\X_2}$ of Hilbert spaces and define $\SmallSysNode$ by \eqref{eq:sysnodecayl}. Then the operator $\sbm{\sqrt2\,I&0\\0&I}-\sbm{0\\A_2}$ maps $\dom{A_{ext}}$ one to one onto $\dom{\SmallSysNode}$ and 
\begin{equation}\label{eq:syscaylopchar}
\begin{aligned}
  \bbm{\AB\\\CD} &= \bbm{\sqrt 2\,A_1\\A_2+\bbm{0&I}} \left(\bbm{\sqrt2\,I&0\\0&I}-\bbm{0\\A_2}\right)^{-1}\quad\text{with}\\
  \dom{\SmallSysNode} &= \left(\bbm{\sqrt2\,I&0\\0&I}-\bbm{0\\A_2}\right)\dom{A_{ext}}.
\end{aligned}
\end{equation}

In particular, if there exist linear operators $A_{12}$ and $A_{21}$,
such that $A_1\sbm{x\\e}=A_{12}\,e$ and
$A_2\sbm{x\\e}=A_{21}x$ for all
$\sbm{x\\e}\in\dom{A_{ext}}$,\footnote{By writing that $A_2\sbm{x\\e}=A_{21}x$ for all
$\sbm{x\\e}\in\dom{A_{ext}}$, we mean that the given operator $A_2$ has the property that $A_2\sbm{x\\e_1}=A_2\sbm{x\\e_2}$ whenever $\sbm{x\\e_1},\sbm{x\\e_2}\in\dom{A_{ext}}$. Then we set $\dom{A_{21}}:=\set{x\bigmid\sbm{x\\e}\in\dom{A_{ext}}}$ and  $A_{21}x:=A_2\sbm{x\\e}$, where $\sbm{x\\e}\in\dom{A_{ext}}$. The same is meant for $A_1$.} 
 then
\begin{equation}\label{eq:syscaylopcharsimp}
\begin{aligned}
  \bbm{\AB\\\CD} &= \bbm{A_{S=I}\,\&\,(\sqrt2\,A_{12})\\(\sqrt2\,A_{21})\,\&\,I},\\
  \dom{\SmallSysNode} &= \bbm{\sqrt2\,I& 0\\-A_{21}&I} \dom{A_{ext}},
\end{aligned}
\end{equation}
where $A_{S=I}:=A_{12}A_{21}$, cf.\ \eqref{eq:ASintro}.
\end{proposition}

The notation in the second part of the proposition is analogous to that in Theorem \ref{thm:parahypconf}.

\begin{proof}
Fix $\sbm{x\\e}\in\dom{A_{ext}}$ arbitrarily and set $\sbm{z\\f}:=A_{ext}\sbm{x\\e}$, $u:=(e-f)/\sqrt2$, and $y:=(e+f)/\sqrt2$. It then follows that $e-A_2\sbm{x\\e}=\sqrt2 u$ and by \eqref{eq:caylandinv} also $y=\sqrt2 A_2 \sbm{x\\e}+u$. Hence
\begin{equation}\label{eq:caylaltrepcalc}
  \left(\bbm{\sqrt2\,I&0\\0&I}-\bbm{0\\A_2}\right)\bbm{x\\e}=\sqrt2\bbm{x\\u}\quad\text{and}\quad \bbm{z\\y}=\bbm{A_1\\\sqrt2\,A_2}\bbm{x\\e}+\bbm{0\\u}.
\end{equation}
We next prove that the operator $\sbm{\sqrt2\,I&0\\0&I}-\sbm{0\\A_2}$ is injective and therefore assume that $\left(\sbm{\sqrt2\,I&0\\0&I}-\sbm{0\\A_2}\right)\sbm{x\\e}=0$. Then $x=0$ and $f=A_2\sbm{0\\e}=e$, which implies that $u=\frac{e-f}{\sqrt 2}=0$, and it follows from \eqref{eq:energyineqpf} that $y=0$, which in turn implies that $e=\frac{y+u}{\sqrt2}=0$. This shows that $\sbm{\sqrt2\,I&0\\0&I}-\sbm{0\\A_2}$ is injective. The domain of this operator is clearly $\dom{A_{ext}}$, and by \eqref{eq:sysnodecayl} its range is $\dom{\SmallSysNode}$. 

Moreover, \eqref{eq:caylaltrepcalc} yields that
$$
\begin{aligned}
  \bbm{\AB\\\CD}\bbm{x\\u} &= \bbm{z\\y} = \bbm{A_1\\\sqrt2\,A_2}\left(\bbm{\sqrt2\,I&0\\0&I}-\bbm{0\\A_2}\right)^{-1}\sqrt2\bbm{x\\u}+\bbm{0\\u} \\
  &= \left(\bbm{\sqrt2\,A_1\\2A_2}+\bbm{0&0\\0&I}\left(\bbm{\sqrt2\,I&0\\0&I}-\bbm{0\\A_2}\right)\right) \\
    &\qquad\qquad\times \left(\bbm{\sqrt2\,I&0\\0&I}-\bbm{0\\A_2}\right)^{-1}\bbm{x\\u}\\
  &= \bbm{\sqrt2\,A_1\\A_2+\bbm{0&I}}\left(\bbm{\sqrt2\,I&0\\0&I}-\bbm{0\\A_2}\right)^{-1}\bbm{x\\u}
\end{aligned}
$$
for all $\sbm{x\\u}\in\dom{\SmallSysNode}$, and the first assertion is
proved. From here \eqref{eq:syscaylopcharsimp} follows easily.
\end{proof}

If one assumes more structure of $A_{ext}$ in the preceding proposition, essentially that $A_{ext}$ is a system node, then one can alternatively obtain the result by flow inversion, using \cite[Thm 6.3.9]{StafBook}. 

We now continue the example in the introduction, where $A_{ext}$ in \eqref{eq:AextIntro} is a skew-adjoint operator that is not a system node.

\begin{example}\label{ex:wavesysnode}
The operator $A_{ext}=\sbm{0&\mathrm{div}\\\mathrm{grad}&0}$ with $\dom{A_{ext}}:= \sbm{H^1_0(\Omega)\\H^{\mathrm{div}}(\Omega)}$ is \emph{not} a system node with input space $L^2(\Omega)^n$, because 
$$
  \set{u\in L^2(\Omega)^n\bigmid \exists x\in L^2(\Omega):~
    \bbm{x\\u}\in\dom{A_{ext}}}=H^{\mathrm{div}}(\Omega),
$$
which is a proper subspace of $L^2(\Omega)^n$, and so condition 4 of Definition \ref{def:sysnode} is violated. Moreover, the ``main operator'' of $A_{ext}$ is zero:
$$
  x \mapsto \bbm{0&\mathrm{div}} \bbm{x\\0} = 0, \quad \bbm{x\\0}\in\dom{A_{ext}},~\text{i.e.,}~ x\in H_0^1(\Omega),
$$
and the ``control operator'' $\mathrm{div}$ is unbounded from $L^2(\Omega)^n$ into $L^2(\Omega)$, and so $A_{ext}$ also fails the standard test that the main operator should be the most unbounded operator of the system node.

Although $A_{ext}$ is not a system node, it is maximal dissipative and closed (even self-adjoint; see \cite[Cor.\ 3.4]{parahyp_divgrad}), and hence the extended Cayley system transform $\SmallSysNode$ of $A_{ext}$ is a system node; see Theorem \ref{thm:stwe}. The state space of $\SmallSysNode$ is $X=L^2(\Omega)$, the input and output spaces are $U=Y=L^2(\Omega)^n$, and according to Proposition \ref{prop:caylalt}, the system node itself is given by:
\begin{equation}\label{eq:waveextcayl}
  \bbm{\AB\\\CD} = \bbm{\Delta \,\&\, \sqrt 2\,\mathrm{div} \\ 
    \sqrt 2\,\mathrm{grad} \,\&\, I}\Bigg|_{\dom{\SmallSysNode}},
\end{equation}
where
\begin{equation}\label{eq:waveextcaylDom}
\begin{aligned}
  \dom{\SmallSysNode} &= \set{\bbm{\sqrt 2\,x\\ e-\Grad x}
    \in\bbm{L^2(\Omega)\\L^2(\Omega)^n}\biggmid 
    \bbm{x\\e}\in \bbm{H^1_0(\Omega)\\H^{\mathrm{div}}(\Omega)}}.
\end{aligned}
\end{equation}
Here the main operator $A$ equals the \emph{Laplacian} $\Delta x:=\Div(\Grad x)$ defined on
$$
  \dom{\Delta} = \set{x\bigmid \bbm{x\\0}\in\dom{\SmallSysNode}}=\set{x\in H_0^1(\Omega)\bigmid \Grad x\in H^{\mathrm{div}}(\Omega)}.
$$
We can confirm that $A$ of $\SmallSysNode$ is the most unbounded operator of $\SmallSysNode$.

The PDE associated to the operator $\SmallSysNode$ in \eqref{eq:waveextcayl}--\eqref{eq:waveextcaylDom} is
\begin{equation}\label{eq:wavetoheat}
  \left\{\begin{aligned}
    \frac{\partial x}{\partial t}(\xi,t) &= \Delta x(\xi,t) + \sqrt2\, \Div u(\xi,t) \\
    y(\xi,t) &= \sqrt2\, \Grad x(\xi,t) + u(\xi,t),\quad \text{a.e.}~\xi\in\Omega,~t\geq0,\\
    x(\xi,0) &= x_0(\xi),\quad \text{a.e.}~\xi\in\Omega,\\   
    x(\xi,t) &= 0,\quad \text{a.e.}~\xi\in\partial\Omega,~t\geq0.
       \end{aligned}\right.
\end{equation}
Thus, the external Cayley system transformation of the wave equation is the heat equation with constant thermal conductivity $\alpha(\cdot)=I$ and control and observation along all of the spatial domain.
\end{example}

In the definition \eqref{eq:ASDef} of $A_S$, we expressed $A_S$ in terms of $A_{ext}$, and we now proceed to express $A_S$ in terms of the transform $\SmallSysNode$. 
Combining \eqref{eq:ASDef} and \eqref{eq:sysnodecayl}, we see that $x\in\dom{A_S}$ and $z=A_Sx$ if and only if
\begin{equation}\label{eq:sysnodeAS}
\begin{aligned}
  \exists f\in\dom S& ,e\in X_2:\quad \bbm{x\\e}\in\dom{A_{ext}},~ \bbm{z\\f}=A_{ext}\bbm{x\\e},~e = Sf\\
  \Longleftrightarrow\qquad& \exists f\in \dom S,e\in X_2:\quad \bbm{x\\(e-f)/\sqrt2}\in\dom{\SmallSysNode},\\
   &\quad \text{and}\quad\bbm{z\\(e+f)/\sqrt2}=\bbm{\AB\\\CD}\bbm{x\\(e-f)/\sqrt2},~ e=Sf\\
  \Longleftrightarrow\qquad& \exists u,y\in X_2:\quad y-u\in\dom S,~\bbm{x\\u}\in\dom{\SmallSysNode},\\
   &\quad \text{and}\quad\bbm{z\\y}=\bbm{\AB\\\CD}\bbm{x\\u},~ \frac{y+u}{\sqrt2}=S\frac{y-u}{\sqrt2}.
\end{aligned}
\end{equation}

Since $\SmallSysNode$ is a well-posed system node, contrary to $A_{ext}$, it now makes sense to write the equation $y+u=S(y-u)$ in the form $u=Ky$ and interpret $K$ as an output feedback operator for $\SmallSysNode$. We next show that $y-u\in\dom S$ and $y+u=S(y-u)$ if and only if $u=Ky$, where
\begin{equation}\label{eq:Kdef}
  K:=(S-I)(S+I)^{-1}.
\end{equation}
We call this $K$ the \emph{operator Cayley transform} of the maximal accretive operator $S$.

It is important to pay attention to the condition $\delta\geq0$ versus
the condition $\delta>0$ in \eqref{eq:accretivity} below. If
$\delta=0$ then $S$ is only accretive, whereas $\delta>0$ implies that
$S$ is \emph{uniformly} accretive. Neither of these conditions alone
implies any kind of maximality; see the second assertion in the following lemma.

\begin{lemma}\label{lem:Scayl1}
The following claims are true:
\begin{enumerate}
\item Let $S$ be a closed and maximal accretive operator on $X_2$. Then $S+I$ has a bounded inverse and the operator $K$ in \eqref{eq:Kdef} is an everywhere-defined contraction on $X_2$, i.e., $\|K\|\leq 1$. 

The contraction $K$ has the additional property that $I-K$ is injective with range dense in $X_2$, and $S$ can be recovered from $K$ using the formula
\begin{equation}\label{eq:SfromK}
  S = (I+K)(I-K)^{-1} \quad\text{with}\quad \dom{S} = \range{I-K}.
\end{equation}
\item If $S$ is a closed and accretive and everywhere-defined operator on $X_2$, then $S$ is bounded and maximal accretive.
\item If $S$ is closed, defined on all of $X_2$, and uniformly accretive, i.e., there exists a $\delta>0$ such that
\begin{equation}\label{eq:accretivity}
  \re \Ipdp{Sf}{f}\geq \delta \|f\|^2,\quad f\in X_2,
\end{equation}
then $K$ in \eqref{eq:Kdef} is a strict contraction: 
$$
  \|K\|\leq \varepsilon <1\quad \text{where}\quad \varepsilon:=\sqrt{ 1 - \frac{4\delta}{\|S+I\|^2}}.
$$ 
\end{enumerate}
\end{lemma}
\begin{proof}
Assertion 2 holds because $S$ is accretive and bounded (by the closed graph theorem), and clearly $S$ has no proper extension to an operator on $X_2$.

Now assume that $S$ is an arbitrary closed and maximal accretive operator on $X_2$. Then $-S$ is closed and maximal dissipative, and hence $-1\in\res{S}$ by the Lumer-Phillips Theorem \ref{thm:lumphi}, and so $S+I$ is boundedly invertible. Moreover, $K$ is a contraction because the accretivity of $S$ implies that for all $y\in\range{S+I}=X_2$:
\begin{equation}\label{eq:Tcontr}
\begin{aligned}
  \|Ky\|^2 -\|y\|^2 &= \Ipdp{(S-I)(S+I)^{-1}y}{(S-I)(S+I)^{-1}y} \\
  &\qquad - \Ipdp{(S+I)(S+I)^{-1}y}{(S+I)(S+I)^{-1}y} \\
  &= -4 \re \Ipdp{S(S+I)^{-1}y}{(S+I)^{-1}y} \leq 0.
\end{aligned}
\end{equation}
It follows directly from $K=(S-I)(S+I)^{-1}$ that $I+K=2S(S+I)^{-1}$ and $I-K=2(S+I)^{-1}$, so that $I-K$ is injective with $\range{I-K}=\dom S$ and $(I+K)(I-K)^{-1}=S$. According to Theorem \ref{thm:lumphi}, $\range{I-K}=\dom S$ is dense in $X_2$, and this finishes the proof of assertion one.

Now assume that $S$ is bounded with $\dom S=X_2$ and $\re \Ipdp{Sf}{f}\geq\delta\|f\|^2$ for some $\delta>0$ and all $f\in X_2$. Then it holds for all $f\in X_2$ that
$$
  \re \Ipdp{Sf}{f}\geq\delta\|f\|^2 \geq \frac{\delta}{\|S+I\|^2}\|S+I\|^2\|f\|^2 \geq \frac{\delta}{\|S+I\|^2}\|(S+I)f\|^2,
$$
and choosing $f:=(S+I)^{-1}y$ for an arbitrary $y\in X_2$, we obtain that
$$
  \frac{\delta}{\|S+I\|^2}\|y\|^2 \leq \re \Ipdp{S(S+I)^{-1}y}{(S+I)^{-1}y}\quad \forall y\in X_2.
$$
Thus we can sharpen \eqref{eq:Tcontr} into
$$
\begin{aligned}
  \frac{\|Ky\|^2}{\|y\|^2} &=  \frac{\|y\|^2-4 \re \Ipdp{S(S+I)^{-1}y}{(S+I)^{-1}y}}{\|y\|^2}
  \leq 1 - \frac{4\delta}{\|S+I\|^2},
\end{aligned}
$$
and therefore $\|K\|\leq \sqrt{ 1 - 4\delta/\|S+I\|^2}<1$, as claimed in assertion 3.
\end{proof}

The following lemma gives a converse to the preceding result:

\begin{lemma}\label{lem:Scayl2}
Assume that $K$ is an everywhere-defined contraction with $I-K$ injective. Then $S$ defined by \eqref{eq:SfromK} is a maximal accretive, in general unbounded but densely defined and closed, operator on $X_2$. 

The operator $S+I$ has a bounded inverse defined on all of $X_2$ and $K$ can be recovered from $S$ using \eqref{eq:Kdef}. Moreover, \eqref{eq:accretivity} holds with
\begin{equation}\label{eq:accdelta}
  \delta := \frac{1-\|K\|^2}{\|I-K\|^2}.
\end{equation}
In particular, if $\|K\|<1$ then $I-K$ has a bounded inverse and $\delta>0$ in \eqref{eq:accdelta}. In this case $S$ is also bounded: $\|S\| \leq (1+\|K\|)/(1-\|K\|)$.
\end{lemma}
\begin{proof}
Assume that $K$ is an arbitrary contraction such that $I-K$ is injective. It follows from \eqref{eq:SfromK} that $S+I=2(I-K)^{-1}$, and $S-I=2K(I-K)^{-1}$. Hence $\range{S+I}=\dom{I-K}=X_2$ and \eqref{eq:Kdef} holds. From \eqref{eq:Kdef} it follows that \eqref{eq:Tcontr} holds, and from \eqref{eq:Tcontr} it in turn follows that for all $f\in\dom S$:
$$
\begin{aligned}
  \re\Ipdp{Sf}{f} &= \frac{\|(S+I)f\|^2 -\|K(S+I)f\|^2}{4} \\
  &\geq \frac{\|(S+I)f\|^2-\|K\|^2 \|(S+I)f\|^2}{4} \\
  &\geq \frac{1-\|K\|^2}{4}\, \|(S+I)f\|^2 \geq \frac{1-\|K\|^2}{4}\, \|2(I-K)^{-1}f\|^2 \\
  &\geq \frac{1-\|K\|^2}{\|I-K\|^2}\, \|I-K\|^2\,\|(I-K)^{-1}f\|^2 \geq \frac{1-\|K\|^2}{\|I-K\|^2}\, \|f\|^2 \geq 0.
\end{aligned}
$$
Thus \eqref{eq:accretivity} holds with $\delta$ in \eqref{eq:accdelta}, and we have showed that $S$ is accretive with the property $\range{S+I}=X_2$. By the Lumer-Phillips Theorem \ref{thm:lumphi}, $S$ is maximal accretive, densely defined, and closed. 

Finally assume that $\|K\|<1$. Then $I-K$ is boundedly invertible, or more precisely, $\|(I-K)^{-1}\|\leq 1/(1-\|K\|)$, as can easily be seen using Neumann series. Thus
$$
  \|S\|=\|(I+K)(I-K)^{-1}\|\leq \|I+K\|\|(I-K)^{-1}\|\leq \frac{1+\|K\|}{1-\|K\|}.
$$
\end{proof}

The following simple observation turns out to be useful:

\begin{corollary}\label{cor:cayl}
Let the operators $S$ and $K$ be related by \eqref{eq:Kdef}--\eqref{eq:SfromK}. Then $u=Ky$ if and only if $y-u\in\dom S$ and $y+u=S(y-u)$.
\end{corollary}
\begin{proof}
Assume that $y-u\in\dom S$ and $y+u=S(y-u)$. Then $(S+I)(y-u)=2y$ and $(S-I)(y-u)=2u$, which implies that $2u=(S-I)(S+I)^{-1}2y=2Ky$. Conversely, if $u=Ky$, then it follows from \eqref{eq:SfromK} that $y-u=(I-K)y\in\dom S$ and $y+u=S(y-u)$.
\end{proof}

The main findings of this section are now collected in the following proposition:

\begin{proposition}\label{prop:Sec2Concl}
Let $A_{ext}$ be a closed and maximal dissipative operator on the pair $\sbm{X_1\\X_2}$ of Hilbert spaces, and let $S$ be a closed and maximal accretive operator on $X_2$. Define $\SmallSysNode$ by \eqref{eq:sysnodecayl} and $K$ by \eqref{eq:Kdef}. Then the following claims are true:
\begin{enumerate}
\item The operator $\SmallSysNode$ is a passive system node with state space $X_1$ and input/output space $X_2$, and $K$ is a contraction on $X_2$. The operator $K$ is a strict contraction if and only if $S$ is bounded and uniformly accretive.

\item The operator $A_S$ defined in \eqref{eq:ASDef} has the alternative representation
\begin{equation}\label{eq:ASaltRep}
\begin{aligned}
  \dom{A_S} &= \bigg\{x\in X_1\bigmid \exists u\in X_2:~\bbm{x\\u}\in\dom{\SmallSysNode},\\
  &\qquad\qquad\qquad\quad y=K[\CD]\bbm{x\\u},~\text{and}~ u=Ky\bigg\},\\
  A_S x &= z,\quad \text{where}~\bbm{z\\y}=\bbm{\AB\\\CD}\bbm{x\\u}~\text{and}~u=Ky.
\end{aligned}
\end{equation}
\end{enumerate}
\end{proposition}
\begin{proof}
Item 1 follows from Theorem \ref{thm:stwe} together with assertions 1 and 3 of Lemma \ref{lem:Scayl1} and Lemma \ref{lem:Scayl2}. The second item holds because the last line of \eqref{eq:sysnodeAS} and \eqref{eq:ASaltRep} are equivalent by Corollary \ref{cor:cayl}.
\end{proof}

In the next section we give some sufficient conditions for $A_S$ to be maximal dissipative by considering $K$ as a static output feedback operator for $\SmallSysNode$; see \eqref{eq:ASaltRep}.

\section{Proof of Theorem \ref{thm:KadmConseq} using feedback theory}\label{sec:scattfeedback}

We first recall some background on feedback in infinite-dimensional systems. We start with a system node $\SmallSysNode$ and a bounded static output feedback operator $K$. We then create a feedback loop from the output $y$ of $\SmallSysNode$ to the input of $K$, and the output of $K$ is fed back into the input $u$ of $\SmallSysNode$. To the input $u$ of $\SmallSysNode$ we also add another external input $v$, and if the resulting mapping $\sbm{A^f\&B^f\\C^f\&D^f}$ from $\sbm{x\\v}$ to $\sbm{z\\y}=\SmallSysNode\sbm{x\\u}$ is again a system node, then we say that $K$ is an admissible static feedback operator for $\SmallSysNode$. The superscript $f$ stands for ``feedback''; see Figure \ref{fig:feedback} for an illustration of $\sbm{A^f\&B^f\\C^f\&D^f}$. Definition \ref{def:feedb} gives the precise definition of the concept which is referred to as \emph{system-node admissibility} in \cite[Def.\ 7.4.2]{StafBook}.

\begin{figure}[!h]
\begin{center}

\ifx\du\undefined
  \newlength{\du}
\fi
\setlength{\du}{5\unitlength}
\begin{tikzpicture}
\pgftransformxscale{1.000000}
\pgftransformyscale{-1.000000}
\definecolor{dialinecolor}{rgb}{0.000000, 0.000000, 0.000000}
\pgfsetstrokecolor{dialinecolor}
\definecolor{dialinecolor}{rgb}{1.000000, 1.000000, 1.000000}
\pgfsetfillcolor{dialinecolor}
\pgfsetlinewidth{0.100000\du}
\pgfsetdash{{1.000000\du}{1.000000\du}}{0\du}
\pgfsetdash{{1.000000\du}{1.000000\du}}{0\du}
\pgfsetroundjoin
{\pgfsetcornersarced{\pgfpoint{0.700000\du}{0.700000\du}}\definecolor{dialinecolor}{rgb}{0.000000, 0.000000, 0.000000}
\pgfsetstrokecolor{dialinecolor}
\draw (-5.000000\du,-4.000000\du)--(-5.000000\du,20.000000\du)--(27.000000\du,20.000000\du)--(27.000000\du,-4.000000\du)--cycle;
}\pgfsetlinewidth{0.100000\du}
\pgfsetdash{}{0pt}
\pgfsetdash{}{0pt}
\pgfsetroundjoin
{\pgfsetcornersarced{\pgfpoint{0.300000\du}{0.300000\du}}\definecolor{dialinecolor}{rgb}{0.000000, 0.000000, 0.000000}
\pgfsetstrokecolor{dialinecolor}
\draw (6.000000\du,4.000000\du)--(6.000000\du,10.000000\du)--(16.000000\du,10.000000\du)--(16.000000\du,4.000000\du)--cycle;
}
\definecolor{dialinecolor}{rgb}{0.000000, 0.000000, 0.000000}
\pgfsetstrokecolor{dialinecolor}
\node[anchor=west] at (15.500000\du,0.000000\du){$\sbm{A^f\&B^f\\C^f\&D^f}$};
\definecolor{dialinecolor}{rgb}{0.000000, 0.000000, 0.000000}
\pgfsetstrokecolor{dialinecolor}
\node[anchor=west] at (7.000000\du,7.000000\du){$\SmallSysNode$};
\pgfsetlinewidth{0.100000\du}
\pgfsetdash{}{0pt}
\pgfsetdash{}{0pt}
\pgfsetroundjoin
{\pgfsetcornersarced{\pgfpoint{0.300000\du}{0.300000\du}}\definecolor{dialinecolor}{rgb}{0.000000, 0.000000, 0.000000}
\pgfsetstrokecolor{dialinecolor}
\draw (8.000000\du,13.000000\du)--(8.000000\du,17.000000\du)--(14.000000\du,17.000000\du)--(14.000000\du,13.000000\du)--cycle;
}
\definecolor{dialinecolor}{rgb}{0.000000, 0.000000, 0.000000}
\pgfsetstrokecolor{dialinecolor}
\node[anchor=west] at (9.250000\du,15.000000\du){$K$};
\pgfsetlinewidth{0.100000\du}
\pgfsetdash{}{0pt}
\pgfsetdash{}{0pt}
\pgfsetmiterjoin
\pgfsetbuttcap
{
\definecolor{dialinecolor}{rgb}{0.000000, 0.000000, 0.000000}
\pgfsetfillcolor{dialinecolor}
\pgfsetarrowsend{stealth}
{\pgfsetcornersarced{\pgfpoint{0.000000\du}{0.000000\du}}\definecolor{dialinecolor}{rgb}{0.000000, 0.000000, 0.000000}
\pgfsetstrokecolor{dialinecolor}
\draw (16.000000\du,9.000000\du)--(22.000000\du,9.000000\du)--(22.000000\du,15.000000\du)--(14.000000\du,15.000000\du);
}}
\pgfsetlinewidth{0.100000\du}
\pgfsetdash{}{0pt}
\pgfsetdash{}{0pt}
\pgfsetmiterjoin
\pgfsetbuttcap
{
\definecolor{dialinecolor}{rgb}{0.000000, 0.000000, 0.000000}
\pgfsetfillcolor{dialinecolor}
\pgfsetarrowsend{stealth}
{\pgfsetcornersarced{\pgfpoint{0.000000\du}{0.000000\du}}\definecolor{dialinecolor}{rgb}{0.000000, 0.000000, 0.000000}
\pgfsetstrokecolor{dialinecolor}
\draw (8.000000\du,15.000000\du)--(0.000000\du,15.000000\du)--(0.000000\du,9.500000\du);
}}
\definecolor{dialinecolor}{rgb}{0.000000, 0.000000, 0.000000}
\pgfsetstrokecolor{dialinecolor}
\node[anchor=west] at (8.000000\du,7.000000\du){};
\pgfsetlinewidth{0.100000\du}
\pgfsetdash{}{0pt}
\pgfsetdash{}{0pt}
\pgfsetbuttcap
{
\definecolor{dialinecolor}{rgb}{0.000000, 0.000000, 0.000000}
\pgfsetfillcolor{dialinecolor}
\pgfsetarrowsend{stealth}
\definecolor{dialinecolor}{rgb}{0.000000, 0.000000, 0.000000}
\pgfsetstrokecolor{dialinecolor}
\draw (-10.000000\du,5.000000\du)--(6.000000\du,5.000000\du);
}
\pgfsetlinewidth{0.100000\du}
\pgfsetdash{}{0pt}
\pgfsetdash{}{0pt}
\pgfsetbuttcap
{
\definecolor{dialinecolor}{rgb}{0.000000, 0.000000, 0.000000}
\pgfsetfillcolor{dialinecolor}
\pgfsetarrowsend{stealth}
\definecolor{dialinecolor}{rgb}{0.000000, 0.000000, 0.000000}
\pgfsetstrokecolor{dialinecolor}
\draw (16.000000\du,5.000000\du)--(34.000000\du,5.000000\du);
}
\definecolor{dialinecolor}{rgb}{0.000000, 0.000000, 0.000000}
\pgfsetstrokecolor{dialinecolor}
\node[anchor=west] at (8.000000\du,7.000000\du){};
\definecolor{dialinecolor}{rgb}{0.000000, 0.000000, 0.000000}
\pgfsetstrokecolor{dialinecolor}
\node[anchor=west] at (9.000000\du,7.000000\du){};
\definecolor{dialinecolor}{rgb}{0.000000, 0.000000, 0.000000}
\pgfsetstrokecolor{dialinecolor}
\node[anchor=west] at (28.000000\du,3.000000\du){$z$};
\definecolor{dialinecolor}{rgb}{0.000000, 0.000000, 0.000000}
\pgfsetstrokecolor{dialinecolor}
\node[anchor=west] at (17.000000\du,17.000000\du){$y$};
\definecolor{dialinecolor}{rgb}{0.000000, 0.000000, 0.000000}
\pgfsetstrokecolor{dialinecolor}
\node[anchor=west] at (-10.000000\du,3.000000\du){$x$};
\definecolor{dialinecolor}{rgb}{0.000000, 0.000000, 0.000000}
\pgfsetstrokecolor{dialinecolor}
\node[anchor=west] at (2.000000\du,17.000000\du){$Ky$};
\definecolor{dialinecolor}{rgb}{0.000000, 0.000000, 0.000000}
\pgfsetstrokecolor{dialinecolor}
\node[anchor=west] at (27.000000\du,4.000000\du){};
\definecolor{dialinecolor}{rgb}{0.000000, 0.000000, 0.000000}
\pgfsetstrokecolor{dialinecolor}
\node[anchor=west] at (18.150000\du,16.850000\du){};
\definecolor{dialinecolor}{rgb}{0.000000, 0.000000, 0.000000}
\pgfsetstrokecolor{dialinecolor}
\node[anchor=west] at (-9.187500\du,3.925000\du){};
\pgfsetlinewidth{0.100000\du}
\pgfsetdash{}{0pt}
\pgfsetdash{}{0pt}
\pgfsetbuttcap
{
\definecolor{dialinecolor}{rgb}{0.000000, 0.000000, 0.000000}
\pgfsetfillcolor{dialinecolor}
\pgfsetarrowsend{stealth}
\definecolor{dialinecolor}{rgb}{0.000000, 0.000000, 0.000000}
\pgfsetstrokecolor{dialinecolor}
\draw (22.000000\du,9.000000\du)--(34.000000\du,9.000000\du);
}
\definecolor{dialinecolor}{rgb}{0.000000, 0.000000, 0.000000}
\pgfsetstrokecolor{dialinecolor}
\node[anchor=west] at (28.000000\du,11.000000\du){$y$};
\pgfsetlinewidth{0.100000\du}
\pgfsetdash{}{0pt}
\pgfsetdash{}{0pt}
\pgfsetbuttcap
{
\definecolor{dialinecolor}{rgb}{0.000000, 0.000000, 0.000000}
\pgfsetfillcolor{dialinecolor}
\pgfsetarrowsend{stealth}
\definecolor{dialinecolor}{rgb}{0.000000, 0.000000, 0.000000}
\pgfsetstrokecolor{dialinecolor}
\draw (-10.000000\du,9.000000\du)--(-0.500000\du,9.000000\du);
}
\definecolor{dialinecolor}{rgb}{0.000000, 0.000000, 0.000000}
\pgfsetstrokecolor{dialinecolor}
\node[anchor=west] at (-10.000000\du,11.000000\du){$v$};
\pgfsetlinewidth{0.100000\du}
\pgfsetdash{}{0pt}
\pgfsetdash{}{0pt}
\pgfsetbuttcap
{
\definecolor{dialinecolor}{rgb}{0.000000, 0.000000, 0.000000}
\pgfsetfillcolor{dialinecolor}
\pgfsetarrowsend{stealth}
\definecolor{dialinecolor}{rgb}{0.000000, 0.000000, 0.000000}
\pgfsetstrokecolor{dialinecolor}
\draw (0.500000\du,9.000000\du)--(6.000000\du,9.000000\du);
}
\definecolor{dialinecolor}{rgb}{0.000000, 0.000000, 0.000000}
\pgfsetstrokecolor{dialinecolor}
\node[anchor=west] at (3.000000\du,11.000000\du){$u$};
\definecolor{dialinecolor}{rgb}{1.000000, 1.000000, 1.000000}
\pgfsetfillcolor{dialinecolor}
\pgfpathellipse{\pgfpoint{0.000000\du}{9.000000\du}}{\pgfpoint{0.500000\du}{0\du}}{\pgfpoint{0\du}{0.500000\du}}
\pgfusepath{fill}
\pgfsetlinewidth{0.100000\du}
\pgfsetdash{}{0pt}
\pgfsetdash{}{0pt}
\definecolor{dialinecolor}{rgb}{0.000000, 0.000000, 0.000000}
\pgfsetstrokecolor{dialinecolor}
\pgfpathellipse{\pgfpoint{0.000000\du}{9.000000\du}}{\pgfpoint{0.500000\du}{0\du}}{\pgfpoint{0\du}{0.500000\du}}
\pgfusepath{stroke}
\definecolor{dialinecolor}{rgb}{0.000000, 0.000000, 0.000000}
\pgfsetstrokecolor{dialinecolor}
\node[anchor=west] at (-4.000000\du,8.000000\du){$+$};
\definecolor{dialinecolor}{rgb}{0.000000, 0.000000, 0.000000}
\pgfsetstrokecolor{dialinecolor}
\node[anchor=west] at (0.000000\du,11.000000\du){$+$};
\end{tikzpicture}

\caption{A standard feedback connection illustrating the closed-loop system node in Definition \ref{def:feedb}.}
\label{fig:feedback}
\end{center}
\end{figure}
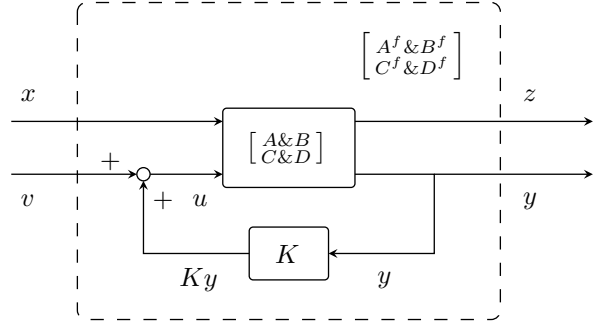

\begin{definition}\label{def:feedb}
Let $\SmallSysNode$ be a system node with input space $U$ and output space $Y$. The bounded linear operator $K$ from $Y$ into $U$ is an \emph{admissible static output feedback operator} for $\SmallSysNode$ if there exists another system node $\sbm{A^f\&B^f\\C^f\&D^f}$ with the same input, state, and output spaces as $\SmallSysNode$, such that the following conditions all hold:
\begin{enumerate}
\item The operator 
\begin{equation}\label{eq:Mdef}
	M:= \bbm{I&0\\0&I}-\bbm{\bbm{0&0}\\K[\CD]} 
\end{equation}
maps $\dom{\SmallSysNode}$ continuously into $\dom{\sbm{A^f\&B^f\\C^f\&D^f}}$.
\item $M$ is invertible and the inverse satisfies 
$$M^{-1} = \bbm{I&0\\0&I}+\bbm{\bbm{0&0}\\K[C^f\&D^f]}.$$
\item The two system nodes are related by 
\begin{equation}\label{eq:closedloop}
	\bbm{A^f\&B^f\\C^f\&D^f}=\SysNode M^{-1}.
\end{equation}
\end{enumerate}
\end{definition}

We refer to $\sbm{A^f\&B^f\\C^f\&D^f}$ in the above result as the \emph{closed-loop} system node corresponding to the coupling of $\SmallSysNode$ and $K$. Note that the operator $M^{-1}$ in Definition \ref{def:feedb} corresponds to the mapping from $\sbm{x\\v}$ to $\sbm{x\\u}$ in Figure \ref{fig:feedback}. 
The $T$-input/output map of Definition \ref{def:DzeroT} plays a key role in determining if a given operator $K$ is an admissible static input/output feedback operator:

\begin{lemma}\label{lem:adm}
Fix $T>0$ arbitrarily and let $\SmallSysNode$ be a \emph{passive} system node with input space $U$, output space $Y$, and $T$-input/output map $\Dfrak_0^T$. Let $K$ be a bounded operator from $Y$ into $U$. Then the following claims are true:
\begin{enumerate}
\item The operator $K$ is an admissible static output feedback operator for $\SmallSysNode$ if $I-K\Dfrak_0^T$ has a bounded inverse in $L^2([0,T];U)$, where $K$ is applied point-wise to a function in $L^2([0,T];Y)$.
\item If $\|K\Dfrak_0^T\|<1$ as an operator on $L^2([0,T];U)$, then $K$ is admissible.
\end{enumerate}
\end{lemma}
\begin{proof}
Since $K$ is applied point-wise, we have that
$$
  \pi_{[0,T]}K\Dfrak\pi_{[0,T]}=K\pi_{[0,T]}\Dfrak\pi_{[0,T]}=K\Dfrak_0^T.
$$
By Remark \ref{rem:StaffansD} combined with \cite[Thm 7.1.8(ii)]{StafBook}, $K$ is admissible even in the well-posed sense described in \cite[Def.\ 7.1.1]{StafBook} if $I-K\Dfrak_0^T$ has a bounded inverse in $L^2([0,T];U)$. By \cite[Thm 7.4.1]{StafBook}, $K$ is then admissible also in the sense of Definition \ref{def:feedb}, and this proves item one. Item two is \cite[Cor.\ 7.1.9(i)]{StafBook}.
\end{proof}

The preceding proof together with Lemma \ref{lem:Scayl1}.3 proves the last claim in Remark \ref{rem:insights}.

We now focus on the sufficient condition 2 in Lemma \ref{lem:adm}. First recall that $\|\Dfrak_0^T\|\leq1$ for a passive system node by the construction of $\Dfrak_0^T$ and that $\|K\|\leq 1$ if $S$ is maximal accretive and closed. Hence, if $A_{ext}$ is maximal dissipative and $S$ is maximal accretive, both being closed, then $\|K\Dfrak_0^T\|\leq\min\set{\|K\|,\|\Dfrak_0^T\|}$, which is strictly less than one if $\|K\|<1$ or $\|\Dfrak_0^T\|<1$. We can now prove the main result of the paper, Theorem \ref{thm:KadmConseq}.

\emph{Proof of Theorem \ref{thm:KadmConseq}.} We assume that $A_{ext}$ is maximal dissipative and closed on $\sbm{X_1\\X_2}$, that $S$ is maximal accretive and closed on $X_2$, and that $K$ is an admissible static feedback operator for $\SmallSysNode$ defined in \eqref{eq:sysnodecayl}. By Theorem \ref{thm:stwe}, $\SmallSysNode$ is a scattering passive system node, and the operator $\sbm{A^f\&B^f\\C^f\&D^f}$ in Definition \ref{def:feedb} is also a system node due to the assumption on $K$. We next compute the main operator $A^f$ of the latter, showing that $A^f=A_S$. 

By \eqref{eq:Adef} and Definition \ref{def:feedb}, $x\in\dom{A^f}$ and $A^fx=z$ if and only if 
$$
\begin{aligned}
  \bbm{x\\0}\in\dom{\bbm{A^f\&B^f\\C^f\&D^f}} &= \left(\bbm{I&0\\0&I}-\bbm{0\\K[\CD]}\right)\dom{\SmallSysNode}\\
  \text{and}\quad z &= \AB \left(\bbm{I&0\\0&I}-\bbm{0\\K[\CD]}\right)^{-1}\bbm{x\\0},
\end{aligned}
$$
which holds if and only if there exist $\sbm{\wtsmash x\\u}\in\dom{\SmallSysNode}$, such that 
\begin{equation}\label{eq:feedbcalc}
  \bbm{x\\0}=\left(\bbm{I&0\\0&I}-\bbm{0\\K[\CD]}\right)\bbm{\wtsmash x\\u}
  \quad\text{and}\quad z=\AB\bbm{\tilde x\\u}.
\end{equation}
The equations \eqref{eq:feedbcalc} clearly hold if and only if 
$$
  \wtsmash x=x \quad\text{and}\quad \bbm{z\\u}=\bbm{\AB\\K[\CD]}\bbm{x\\u},
$$
and summarizing, we find that $x\in\dom{A^f}$ and $A^fx=z$ if and only if
\begin{equation}\label{eq:sysnodeAK}
  \exists u\in X_2:\quad \bbm{x\\u}\in\dom{\SmallSysNode}, ~u = K[\CD]\bbm{x\\u},~z= \AB\bbm{x\\u}.
\end{equation}
By \eqref{eq:ASaltRep}, \eqref{eq:sysnodeAK} is equivalent to $x\in\dom{A_S}$ and $z=A_Sx$. Hence $A^f=A_S$.

Now we prove that $A_S$ generates a contraction semigroup on $X_1$. According to Definitions \ref{def:sysnode} and \ref{def:feedb}, the operator $A^f=A_S$ generates a $C_0$-semigroup. By the Hille-Yosida Theorem \cite[Thm 2.1.12]{CuZwBook}, there exists some $\omega\in\cplus\cap\res {A^f}$, and since $A_S$ is dissipative by \eqref{eq:ASdiss}, we have that $A_S$ generates a contraction semigroup by the Lumer-Phillips theorem \ref{thm:lumphi}.

It now only remains to point out that $K$ is admissible if $S$ is bounded and uniformly accretive, and this follows from Proposition \ref{prop:Sec2Concl}.1, Lemma \ref{lem:adm}, and $\|K\Dfrak_0^T\|\leq \|K\|<1$. $\square$

The following simple example shows that admissibility of $K$ is not necessary for $A_S$ to generate a contraction semigroup:

\begin{example}\label{ex:mainconvfalse}
Take $X_1=X_2=\C$, $A_{ext}=\sbm{0&0\\0&i}$, and $S=i$. Then $\SmallSysNode=\sbm{0&0\\0&-i}$ and $K=i$, so that $M=\sbm{1&0\\0&0}$ which is not injective. Hence $K$ is not admissible, but by \eqref{eq:ASintro} we have $A_S=0$ which nevertheless generates the constant semigroup on $\C$.
\end{example}

In the introduction we proved that the heat equation \eqref{eq:heatND} is associated to a contraction semigroup using the knowledge that the wave equation \eqref{eq:waveintro} is associated to a contraction semigroup. In the case where the thermal diffusivity $\alpha(\cdot)$ is constantly $I$, we obtain $S=I$ which gives $K=0$. In the notation of Definition \ref{def:feedb}, we thus have that $M^{-1}=\sbm{I&0\\0&I}$ and hence $\sbm{A^f\&B^f\\C^f\&D^f}=\SmallSysNode$. Comparing \eqref{eq:heatND} to \eqref{eq:wavetoheat}, we can confirm that in this example indeed $A_S=A^f=A=\Delta$. 

In the next section we study two more examples that fall under Theorem \ref{thm:KadmConseq}. Now we present a list of sufficient conditions on $\SmallSysNode$ for $\|\Dfrak_0^T\|<1$ to hold.

\begin{proposition}\label{prop:scattpass}
Assume that $A_{ext}$ is maximal dissipative and closed. Define $\SmallSysNode$ by \eqref{eq:sysnodecayl}. If at least one of the following conditions is satisfied for some $T>0$, then $\|\Dfrak_0^T\|<1$:
\begin{enumerate}
\item There exist $T>0$ and $N_T<1$, such that it for all classical trajectories with initial state $x(0)=0$, input signal $u(\cdot)$, and output signal $y(\cdot)$ holds that
\begin{equation}\label{eq:unifscatt}
  \int_0^T\|y(t)\|_ Y^2\ud t \leq N_T \int_0^T\|u(t)\|_ U^2\ud t.
\end{equation}
\item For some $T>0$, some $\varepsilon>0$, and all classical trajectories with input signal $u(\cdot)$ and state trajectory $x(\cdot)$ satisfying $x(0)=0$, it holds that
\begin{equation}\label{eq:unifstate}
  \|x(T)\|_X^2 \geq \varepsilon \int_0^T\|u(t)\|_ U^2\ud t.
\end{equation}
\item The system node $\SmallSysNode$ has a delay $\tau>0$ from input to output, i.e., all classical trajectories $(u,x,y)$ with initial state $x(0)=0$  satisfy $\pi_{[0,\tau)}y=0$.
\end{enumerate}
In fact, assumptions 2 and 3 both imply that assumption 1 is satisfied, with $N_T=1-\varepsilon$, and $T:=\tau$, $N_\tau=0$, respectively.
\end{proposition}
\begin{proof}
Combining \eqref{eq:unifscatt} with the denseness of $\Uscr_0^T$ in $L^2([0,T];U)$, see \eqref{eq:classinputdense}, we obtain that $\|\Dfrak_0^T\|\leq N_T<1$. If \eqref{eq:unifstate} holds, then \eqref{eq:unifscatt} holds with $N_T:=1-\varepsilon$, according to \eqref{eq:scattpass}. Finally, if assumption 3 holds, then $\int_0^\tau \|u(t)\|\ud t=0$ for all classical trajectories with $x(0)=0$, so \eqref{eq:unifscatt} holds with $T:=\tau$ and $N_T:=0$.
\end{proof}

By Proposition \ref{prop:DzeroTdecreases1}, it is enough to check the conditions in Proposition \ref{prop:scattpass} for small $T$. The condition \eqref{eq:unifstate} implies that the input-to-state map
$u\mapsto x(T)$, $x(0)=0$, is injective. This condition seems quite
rare; it does not hold for for any finite-dimensional system, since
the input-to-state map maps the dense subspace $\Uscr_0^T$ of
$L^2([0,T];U)$ into the finite-dimensional state space. The condition
\eqref{eq:unifstate} does, however, hold with $\varepsilon=1$ if $A$
generates the outgoing shift on the right half-line with input $u$ at the
boundary $\xi=0$.

\begin{proposition}
Let $A_{ext}$ and $\SmallSysNode$ be as in Proposition \ref{prop:scattpass}, and assume that condition 1 in that proposition holds. Then $A_{ext}$ is in fact a well-posed system node which is in addition \emph{impedance passive}, i.e.,
\begin{equation}\label{eq:imppasv}
	\re\Ipdp zx_{X_1} \leq \re\Ipdp fe_{X_2},\qquad \bbm{x\\e}\in\dom {A_{ext}},~ \bbm{z\\f}=A_{ext}\bbm{x\\e}.
\end{equation}
\end{proposition}

\begin{proof}
The operator $\SmallSysNode$ is a well-posed system node by Theorem \ref{thm:stwe}. By Proposition \ref{prop:scattpass} it holds that $\|\Dfrak_0^T\|<1$ and by Lemma \ref{lem:adm}, $-I$ is then a well-posed-admissible static feedback operator \cite[Def.\ 7.1.1]{StafBook} of 
\begin{equation}\label{eq:caylnode}
	\bbm{\AB\\\CD} = \bbm{\sqrt 2\,A_1\\A_2+\bbm{0&I}} \left(\bbm{\sqrt2\,I&0\\0&I}-\bbm{0\\A_2}\right)^{-1};
\end{equation}
see Proposition \ref{prop:caylalt} (here $A_{ext}=\sbm{A_1\\A_2}$). Using Definition \ref{def:feedb}, we calculate the corresponding well-posed closed-loop system node by inserting \eqref{eq:caylnode} into \eqref{eq:Mdef}:
$$
	M = \bbm{\sqrt 2 I&0\\0&2I}\left( \bbm{\sqrt 2 I&0\\0&I}-\bbm{0\\A_2} \right)^{-1}.
$$
Using this and \eqref{eq:caylnode} in \eqref{eq:closedloop}, one then obtains
\begin{equation}\label{eq:closedloop2}
	\bbm{A^f&B^f\\C^f&D^f}=\bbm{A_1\\\displaystyle\frac{1}{\sqrt2}\left(A_2+\bbm{0&I}\right)}\bbm{I&0\\0&\frac{1}{\sqrt2}I}.
\end{equation}

It is now established that $\sbm{A^f&B^f\\C^f&D^f}$ satisfies the conditions in Definition \eqref{def:sysnode} and that for any fixed $T>0$ there exists an $M_T\geq0$, such that \eqref{eq:scattpass} holds for all trajectories of $\sbm{A^f&B^f\\C^f&D^f}$. We leave it for the reader to verify that this implies that $\sbm{A_1\\A_2}$ also satisfies the conditions  in Definition \eqref{def:sysnode} and that for the same $T$ and all trajectories of $\sbm{A_1\\A_2}$, the inequality \eqref{eq:scattpass} holds with $4M_T$ instead of $M_T$.

The inequality \eqref{eq:imppasv} is obtained by substituting $u=(e-f)/\sqrt2$ and $y=(e+f)/\sqrt2$ into \eqref{eq:scattpassdiff}, and this completes the proof.
\end{proof}

The preceding result was kindly pointed out to us by the anonymous referee. It says that Proposition \ref{prop:scattpass} is only applicable to well-posed systems. Here is furthermore an example showing that Proposition \ref{prop:scattpass} fails to cover the (well-posed) wave equation:

\begin{example}\label{ex:schronot}
Unfortunately, the external Cayley system transform \eqref{eq:waveextcayl}--\eqref{eq:waveextcaylDom} of the wave equation \eqref{eq:waveintro} does not satisfy \eqref{eq:unifscatt} for any $N_T<1$, because $\|\Dfrak_0^T\|=1$. 

Indeed, since $\Omega$ is a \emph{bounded} Lipschitz domain, we can choose a non-zero constant input signal $u(\xi,t):=u_0\in \R^n$ for all $t\geq0$ and almost every $\xi\in\Omega$. With this input signal and $x_0=0$ in \eqref{eq:wavetoheat}, we obtain that $\partial x(\xi,t)/\partial t=0$ for every $t\geq0$ and almost every $\xi\in\Omega$, and so the state stays at zero: $x(\cdot,t)=0$ in $L^2(\Omega)$ for all $t\geq0$. Hence the corresponding output is $y(\xi,t)=u(\xi,t)=u_0$ for all $t\geq0$ and almost every $\xi\in\Omega$. This implies that 
$$
  \int_0^T\|y(t)\|_{L^2(\Omega)^n}^2\ud t = \int_0^T\|u(t)\|_{L^2(\Omega)^n}^2\ud t = T\,\mathrm{vol}\,\Omega\,\|u_0\|^2_{\R^n}>0
$$
for all $T>0$, and so $N_T=1$ is the smallest possible choice in \eqref{eq:unifscatt} for all $T>0$.
\end{example}

\section{Wave equations with damping along the spatial domain}\label{sec:plates}

In this section we use the approach outlined in the introduction to show that the wave equation with viscous damping and the wave equation with structural damping, both with the damping along the spatial domain, are also associated to contraction semigroups. We shall make use of the following operators $A_{ext}$.

\begin{proposition}\label{prop:skewadjwave}
For a bounded Lipschitz domain $\Omega\subset\R^n$, the following operators are skew adjoint (and closed) on $L^2(\Omega)^{2n+1}$ and $L^2(\Omega)^{n+2}$, respectively:
\begin{equation}\label{eq:AextS}
\begin{aligned}
  A_{ext,s} :=& \bbm{0&\mathrm{div}\bbm{I&I}\\\bbm{I\\I}\mathrm{grad}&\bbm{0&0\\0&0}}\quad\text{with}\\
    \dom{A_{ext,s}} :=& \set{\bbm{x_1\\x_2\\e}\in\bbm{H^1_0(\Omega)\\L^2(\Omega)^n\\L^2(\Omega)^n}
      \bigmid x_2+e\in H^{\mathrm{div}}(\Omega)},\quad\text{and}
\end{aligned}
\end{equation}
\begin{equation}\label{eq:AextV}
  A_{ext,v} := \bbm{0&\mathrm{div}&I\\\mathrm{grad}&0&0\\-I&0&0}\quad\text{with}\quad
  \dom{A_{ext,v}} := \bbm{H^1_0(\Omega)\\H^{\mathrm{div}}(\Omega)\\L^2(\Omega)}.
\end{equation}
\end{proposition}
\begin{proof}
By Theorem 6.2 in \cite{parahyp_report}, $\mathrm{grad}|_{H^1_0(\Omega)}^*=-\mathrm{div}|_{H^{\mathrm{div}}(\Omega)}$. Combining this with Lemma \ref{lem:adjoints} below, we obtain that
$$
\begin{aligned}
  A_{ext,s}^* &= \bbm{0&\mathrm{div}\bbm{I&I}\\\bbm{I\\I}\mathrm{grad}|_{H^1_0(\Omega)}&\bbm{0&0\\0&0}}^* \\
    &= \bbm{0&\left(\bbm{I\\I}\,\mathrm{grad}|_{H^1_0(\Omega)}\right)^*\\\left(\mathrm{div}\,\bbm{I&I}\right)^*&\bbm{0&0\\0&0}}\\
    &= \bbm{0&-\mathrm{div}\bbm{I&I}\\\bbm{I\\I}(-\mathrm{grad}|_{H^1_0(\Omega)})&\bbm{0&0\\0&0}} = - A_{ext,s},
\end{aligned}
$$ 
where we used that the diagonal blocks are zero operators and that the domain of $A_{ext,s}$ decomposes into the product of $\dom{\sbm{I\\I}\mathrm{grad}|_{H^1_0(\Omega)}}$ and $\dom{\mathrm{div}\bbm{I&I}}$.

We also have that $(Q+R)^*=Q^*+R^*$ if $R$ is bounded and everywhere defined. From this it immediately follows that 
$$
  A_{ext,v}=\bbm{0&\mathrm{div}&0\\\mathrm{grad}|_{H^1_0(\Omega)}&0&0\\0&0&0}
    +\bbm{0&0&I\\0&0&0\\-I&0&0}
$$
is skew-adjoint.
\end{proof}

We remark that \cite[Thm 6.2]{parahyp_report} allows a wide range of boundary conditions in addition to those used above for $A_{ext,v}$ and $A_{ext,s}$.

\subsection{Wave equations with viscous damping}\label{sec:viscous}

We first consider the wave equation with \emph{viscous damping} on a bounded Lipschitz domain $\Omega$:
\begin{equation}\label{eq:Wave1orderVisc}
  \left\{
    \begin{aligned}
        \rho(\xi)\frac{\partial^2 x}{\partial t^2} (\xi,t) &= 
	  \Div\big(T(\xi)\,\Grad x(\xi,t)\big)-k_v(\xi)\frac{\partial x}{\partial t}(\xi,t),
	  \quad \xi\in\Omega,~t\geq0,\\
	x(\xi,0) &= x_0(\xi),\quad \frac{\partial x(\xi,0)}{\partial t} = z_0(\xi),
	  \quad \xi\in\Omega,\\   
	\frac{\partial x(\xi,t)}{\partial t} &= 0,\quad \xi\in\partial\Omega,~t\geq0,
    \end{aligned}
  \right.
\end{equation}
where $x(\xi,t)$ is the deflection at point $\xi$ and time $t$, $\rho(\cdot)$ is the mass density, $T(\cdot)$ is Young's modulus, and $k_v(\cdot)$ is the scalar viscous damping coefficient. For physical reasons $\rho(\cdot),k_v(\cdot)\in L^\infty(\Omega)$ take real values and $T(\cdot)\in L^\infty(\Omega)^{n\times n}$ with $T(\xi)^*=T(\xi)$ for almost all $\xi\in\Omega$. We make the additional assumption that $\rho(\cdot)$, $T(\cdot)$, and $k_v(\cdot)$ are bounded away from zero, i.e., that there exists a $\delta>0$, such that $\rho(\xi)\geq\delta$, $k_v(\xi)\geq\delta$, and $T(\xi)\geq \delta I$ for almost all $\xi\in\Omega$. This implies that the operators of multiplication by $\rho(\cdot)$, $T(\cdot)$, and $k_v(\cdot)$ are self-adjoint, bounded, and uniformly accretive on $L^2(\Omega)$, $L^2(\Omega)^{n\times n}$, and $L^2(\Omega)$, respectively. 

The following multiplication operator is also bounded, everywhere defined, self-adjoint, and uniformly accretive on $X_1:=\sbm{L^2(\Omega)\\L^2(\Omega)^n}$:
\begin{equation}\label{eq:HscrDef}
  \Hscr x := \xi\mapsto\bbm{1/\rho(\xi)&0\\0&T(\xi)}x(\xi), \quad 
    \xi\in\Omega,~x\in X_1.
\end{equation}
This operator defines an alternative, but equivalent, inner product on $X_1$ through $\Ipdp{z_1}{z_2}_{\Hscr}:=\Ipdp{\Hscr z_1}{z_2}$, where $\Ipdp{\cdot}{\cdot}$ is the standard inner product on $\sbm{L^2(\Omega)\\L^2(\Omega)^n}$. We denote $X_1$ equipped with the inner product $\Ipdp{\cdot}{\cdot}_\Hscr$ by $X_\Hscr$, and by $X_1$ we mean $X_1$ equipped with the standard $L^2(\Omega)^{n+1}$-inner product.

We can write \eqref{eq:Wave1orderVisc} in the first-order abstract ODE form
\begin{equation}\label{eq:visc1ord}
\left\{\begin{aligned}
  \DD t\bbm{\rho(\cdot)\frac{\ud x(t)}{\ud t}\\\Grad x(t)} &= \bbm{0&\mathrm{div}\\\mathrm{grad}&0} \Hscr
    \bbm{\rho(\cdot)\frac{\ud x(t)}{\ud t}\\\Grad x(t)} + \bbm{I\\0}e(t),\\ 
    e(t) &= k_v(\cdot)\bbm{-I&0}\Hscr\bbm{\rho(\cdot)\frac{\ud x(t)}{\ud t}\\\Grad x(t)},
      \quad t\geq0,
\end{aligned}\right.
\end{equation}
whose state is $\sbm{\rho(\cdot)\frac{\ud x(t)}{\ud t}\\\Grad x(t)}$. The natural state space is $X_\Hscr:=\sbm{L^2(\Omega)\\L^2(\Omega)^n}$ (with the $\Hscr$-inner product induced by $\Hscr$ in \eqref{eq:HscrDef}).

Following Section 2 in \cite{ParaHyp}, we define $X_2:=L^2(\Omega)$, and and we choose $S_v$ to be the bounded and uniformly accretive multiplication operator 
$$
  S_vx:=\xi\mapsto k_v(\xi)\,x(\xi)\quad\text{on}\quad X_2.
$$
This allows us to rewrite \eqref{eq:visc1ord} as
\begin{equation}\label{eq:visc1ord2}
  \DD t\bbm{\rho(\cdot)\frac{\ud x(t)}{\ud t}\\\Grad x(t)} = A_{S,v}\Hscr
    \bbm{\rho(\cdot)\frac{\ud x(t)}{\ud t}\\\Grad x(t)},\quad t\geq0,
\end{equation}
where, using \eqref{eq:ASDef},
$$
  A_{S,v}=\bbm{-S_v&\Div\\\Grad&0}\quad\text{with}\quad \dom{A_{S,v}}=\bbm{H^1_0(\Omega)\\H^{\mathrm{div}}(\Omega)}.
$$ 
By the following result (see {\cite[Lem.\ 7.2.3]{JaZwBook}}), \eqref{eq:visc1ord2} is associated to a contraction semigroup on $X_\Hscr$ if and only if $A_{S,v}$ is maximal dissipative on $X_1$ (with the standard $L^2(\Omega)^{n+1}$-inner product):

\begin{lemma}\label{lem:energynorm}
Let $\Hscr$ be a bounded, self-adjoint, and uniformly accretive operator on a Hilbert space $X_1$. Then a linear operator $A$ generates a contraction semigroup (a unitary group) on $X_1$ if and only if the operator $A\Hscr$ with domain $\dom{A\Hscr} =\set{x\in X_1\mid \Hscr x\in\dom A}$ generates a contraction semigroup (unitary group) on $X_\Hscr$.
\end{lemma}

Since $S_v$ is bounded and uniformly accretive and $A_{ext,v}^*=-A_{ext,v}$ by Proposition \ref{prop:skewadjwave}, $A_{S,v}$ is maximal dissipative on $X_1$; see Theorem \ref{thm:KadmConseq}. Therefore \eqref{eq:Wave1orderVisc} is governed by a contraction semigroup on $X_\Hscr$ in the following sense: The PDE \eqref{eq:Wave1orderVisc}
has a unique solution $x$ for every initial condition, and for this solution the family of mappings
$$
  \bbm{\rho(\cdot)z_0(\cdot)\\\Grad x_0(\cdot)} \mapsto \bbm{\rho(\cdot)\frac{\partial x}{\partial t}(\cdot,t)\\\Grad x(\cdot,t)},\quad t\geq0,
$$
is a contraction semigroup on $X_\Hscr$, cf.\ \eqref{eq:visc1ord2}.

It follows from Proposition \ref{prop:caylalt} that the external Cayley system transform of $A_{ext,v}$ is
\begin{equation}\label{eq:caylvisc}
\begin{aligned}
  \SmallSysNode_v &:= \bbm{\bbm{-I&\mathrm{div}\\\mathrm{grad}&0}\,\&\,\bbm{\sqrt2\,I\\0}\\
    \bbm{-\sqrt2\,I&0}\,\&\,I} \quad\text{with}\\
  \dom{\SmallSysNode_v} &:= \set{\bbm{\sqrt2\,x_1\\\sqrt2\,x_2\\e-\Grad x_1}
    \in\bbm{H^1_0(\Omega)\\H^{\mathrm{div}}(\Omega)\\L^2(\Omega)}\bigmid e\in L^2(\Omega)}.
\end{aligned}
\end{equation}

It is a consequence of the following result that $\|\Dfrak_0^T\|=1$ for the system node \eqref{eq:caylvisc}, and hence Proposition \ref{prop:scattpass} is not applicable to the wave equation with viscous damping:

\begin{proposition}\label{prop:DzeroTdecreases2}
For a well-posed system $\SmallSysNode$ with input space $U$ and output space $Y$, the following claims are true:
\begin{enumerate}
\item Let $D:U\to Y$ be bounded and let $\Lambda_D^T$ denote the bounded operator from $L^2([0,T];U)$ to $L^2([0,T];Y)$ of point-wise multiplication by $D$. If $\lim_{T\to0^+} \|\Lambda_D^T-\Dfrak_0^T\|=0$, where $\|\cdot\|$ denotes the norm of bounded linear operators from $L^2([0,T];U)$ to $L^2([0,T];Y)$, then $\|\Dfrak_0^T\|\geq \|D\|$ for all $T>0$.

\item Denote the state space of $\SmallSysNode$ by $X$ and assume that there exist bounded operators $B:U\to X$, $C:X\to Y$, and $D:U\to Y$, such that $\sbm{\AB\\\CD}=\sbm{A&B\\C&D}\big|_{\dom{\SmallSysNode}}$. Then there for every $T_0>0$ exists a constant $k_0\geq1$, such that $\|\Dfrak_0^T-\Lambda_D^T\|\leq k_0T$ for all $0<T\leq T_0$. In particular, Assertion (1) applies, so that $\|\Dfrak_0^T\|\geq\|D\|$.
\end{enumerate}
\end{proposition}

One uses the triangle inequality to establish the first assertion and the second assertion is proved by using a standard convolution estimate on the variation of constants formula.

\subsection{Structural damping}

Using exactly the same argument as in Section \ref{sec:viscous}, we can prove that the wave equation with \emph{structural damping},
\begin{equation}\label{eq:Wave1orderStruct}
  \left\{
    \begin{aligned}
        \rho(\xi)\frac{\partial^2 x}{\partial t^2} (\xi,t) &= 
	  \Div\big(T(\xi)\,\Grad x(\xi,t)\big)+\Div\Big(k_s(\xi)\,\Grad \frac{\partial x}{\partial t}(\xi,t)\Big),
	\\
	x(\xi,0) &= x_0(\xi),\quad \frac{\partial x(\xi,0)}{\partial t} = z_0(\xi),
	  \quad \xi\in\Omega,\\   
	\frac{\partial x(\xi,t)}{\partial t}&= 0,\quad \xi\in\partial\Omega,~t\geq0,
    \end{aligned}
  \right.
\end{equation}
is also associated to a contraction semigroup on $X_\Hscr$. We make the same assumptions on $\rho(\cdot)$ and $T(\cdot)$ as in \eqref{eq:Wave1orderVisc}, so that $\Hscr$ in \eqref{eq:HscrDef} again defines the inner product of a Hilbert space $X_\Hscr$. Moreover, we assume that $k_s(\cdot)\in L^\infty(\Omega)^{n\times n}$ satisfies $k_s(\xi)+k_s(\xi)^*\geq\delta I$ for some $\delta>0$ and almost every $\xi\in\Omega$, so that the multiplication operator 
$$
  S_sx:=\xi\mapsto k_s(\xi)\,x(\xi)\quad\text{on}\quad X_2:=L^2(\Omega)^n
$$ 
is bounded, everywhere defined, and uniformly accretive. As extended operator we use $A_{ext,s}$ in \eqref{eq:AextS}, and we can use Theorem \ref{thm:KadmConseq} and Lemma \ref{lem:energynorm} to conclude that \eqref{eq:Wave1orderStruct} is governed by a contraction semigroup on $X_\Hscr$. The corresponding operator $A_S$ is
$$
\begin{aligned}
  A_{S,s} &= \bbm{\Div\bbm{S_s\,\Grad&I}\\\bbm{\Grad&0}}, \\
  \dom{A_{S,s}} &= \set{\bbm{x_1\\x_2}\in\bbm{H^1_0(\Omega)\\L^2(\Omega)^n}\bigmid S_s\,\Grad x_1+x_2\in H^{\mathrm{div}}(\Omega)}.
\end{aligned}
$$ 

By Proposition \ref{prop:caylalt}, the external Cayley system transform of $A_{ext,s}$ is 
$$
\begin{aligned}
  \SmallSysNode_s &:=\bbm{\bbm{\Delta&\mathrm{div}\\\mathrm{grad}&0}\,\&\,\bbm{\sqrt2\,\mathrm{div}\\0}\\
    \bbm{\sqrt2\,\mathrm{grad}&0}\,\&\,I}\quad\text{with}\\
  \dom{\SmallSysNode_s} &:=\left\{\bbm{\sqrt2\,x_1\\\sqrt2\,x_2\\e-\Grad x_1}\in\bbm{H^1_0(\Omega)\\L^2(\Omega)\\L^2(\Omega)}\biggmid \right.\\
    &\qquad\qquad\qquad e\in L^2(\Omega),~ x_2+e\in H^{\mathrm{div}}(\Omega)\bigg\}.
\end{aligned}
$$
Hence the main operator $A$ is given by (see \eqref{eq:Adef})
$$
\begin{aligned}
  A\bbm{x_1\\x_2} &= \bbm{\mathrm{div}(\Grad x_1+x_2)\\\Grad x_1},\\
  \dom A &= \set{\bbm{x_1\\x_2}\in\bbm{H^1_0(\Omega)\\L^2(\Omega)}\bigmid \Grad x_1+x_2\in H^{\mathrm{div}}(\Omega)}.
\end{aligned}
$$
Here the control and observation operators are unbounded, so Proposition \ref{prop:DzeroTdecreases2} is not applicable. However, the technique in Example \ref{ex:schronot} can easily be adapted to show that $\|\Dfrak_0^T\|=1$ also in this case, so application of Proposition \ref{prop:scattpass} is excluded.

One can also treat wave equations with both viscous and structural damping. Indeed, from the proof of Proposition \ref{prop:skewadjwave} it follows that the operator
\begin{equation}\label{eq:AextVS}
\begin{aligned}
  A_{ext,vs} :=& \bbm{0&\mathrm{div}\bbm{I&I}&I\\\bbm{I\\I}\mathrm{grad}&\bbm{0&0\\0&0}&\bbm{0\\0}\\-I&\bbm{0&0}&0}, \\
    \dom{A_{ext,vs}} :=& \set{\bbm{x_1\\x_2\\e_1\\e_2}\in
      \bbm{H^1_0(\Omega)\\L^2(\Omega)^n\\L^2(\Omega)^n\\L^2(\Omega)}
      \bigmid x_2+e_1\in H^{\mathrm{div}}(\Omega)},
\end{aligned}
\end{equation}
is skew-adjoint (in particular closed) on $L^2(\Omega)^{2n+2}$. This operator can be associated to a wave equation with both viscous and structural damping by defining $S_{vs}$ to be the operator of multiplication by $\sbm{k_s(\cdot)&0\\0&k_v(\cdot)}$ on $\sbm{L^2(\Omega)^n\\L^2(\Omega)}$. From here we can, however, not immediately deduce that the PDEs \eqref{eq:Wave1orderVisc} and \eqref{eq:Wave1orderStruct} are associated to contraction semigroups by setting $k_v(\cdot):=0$ or $k_s(\cdot):=0$, because $S_{vs}$ is no longer uniformly accretive in that case.

\section{Degenerate parabolic equations}\label{sec:degenerate}

In \cite{ParaHypConf} it is shown how well-posedness of the heat equation (\ref{eq:heatND}) can be obtained from the well-posedness of the associated wave equation
(\ref{eq:waveintro}) by means of Theorem \ref{thm:parahypconf}. In
this section we show that Theorem \ref{thm:KadmConseq} allows this
same approach to be extended to degenerate parabolic PDEs, see e.g.\
\cite{CaMePa98,EnNa00,PazyBook}. In a degenerate parabolic equation
the physical parameter, such as $\alpha$ in equation (\ref{eq:heatND}), may
become zero at the boundary of the spatial domain.

Let $H^{\mathrm{div}}_0(\Omega)$ denote the closure in
  $H^{\mathrm{div}}(\Omega)$ of the set of all functions in $C^\infty(\Omega)^n$
  with support contained in the open set
  $\Omega$. This equals the set of all functions in
  $H^{\mathrm{div}}(\Omega)$ for which the normal trace map is
  zero; see \cite[Thm 5.4.2]{parahyp_report} or \cite[Thm I.2.6]{GiRaBook}. Let $K$ be a linear operator which maps
  $H^{\mathrm{div}}_0(\Omega)$ boundedly into $U$, where $U$ is any Hilbert space. In addition assume that the operator $\sbm{\mathrm{div}\\-K}$ with domain $H^{\mathrm{div}}_0(\Omega)$ is closed as an unbounded operator $L^2(\Omega)^n\to\sbm{L^2(\Omega)\\U}$.

Now set $H:=L^2(\Omega)$, $E:=L^2(\Omega)^n$, $E_0:=H^{\mathrm{div}}_0(\Omega)$, $L:=-\mathrm{div}\big|_{E_0}$, $G:=0$. Denoting the dual of $E_0$ with pivot space $L^2(\Omega)^n$ by $E_0'$, we obtain that $L^*=\mathrm{grad}: L^2(\Omega)\to E_0'$ is bounded. It follows from \cite[Thm 1.1]{StWe12b} and Definition \ref{def:sysnode} that the following operator generates a contraction semigroup on $\sbm{L^2(\Omega)\\L^2(\Omega)^n}$:
\begin{equation}
    \label{eq:9}
\begin{aligned}
    A_{ext} &= \left[ \begin{matrix} 0 & \mathrm{div} \\  \mathrm{grad} &-K^*K \end{matrix}\right]
\qquad\text{with domain} \\
  \dom{A_{ext}} &= \left\{ \left[ \begin{matrix} x_1\\ x_2 \end{matrix} \right] \in \left[ \begin{matrix} L^2(\Omega)\\ H^{\mathrm{div}}_0(\Omega) \end{matrix} \right] \bigmid \mathrm{grad}\ x_1 -K^*K \, x_2 \in L^2(\Omega)^n\right\}.
\end{aligned}
\end{equation}

Next we apply Theorem \ref{thm:KadmConseq} with $S$ bounded on $E=L^2(\Omega)^n$ (satisfying the conditions of item 3) to $A_{ext}$. We find that $A_S$ generates a contraction semigroup on $L^2(\Omega)$, where $A_S$ is the mapping from $x_1$ to $z_1$ in
\begin{align}
  \label{eq:9.6}
  \left\{\begin{aligned} z_1&=\Div x_2 \\ z_2&=\Grad x_1-K^*Kx_2 \\ x_2&=S z_2 \end{aligned}\right.\quad\iff\quad
  \left\{\begin{aligned} z_1&=\Div x_2 \\ (S^{-1} +K^*K)x_2&=\Grad x_1 \\ z_2&=S^{-1}x_2 \end{aligned}\right..
\end{align}
Since $E_0'$ is the dual of $E_0$ with pivot space $E$, we can regard $S^{-1}$ as a bounded mapping from $E_0$ into $E_0'$ in \eqref{eq:9.6}.  Furthermore, for $x_2\in E$ we have by item 3 of Theorem \ref{thm:KadmConseq} that, with $\tilde{x}_2=S^{-1}x_2$,
$$
	\re\Ipdp{S^{-1}x_2}{x_2}_{E} = \re\Ipdp{ \tilde{x}_2}{S\tilde{x}_2}_{E}  \geq \delta\,\|\tilde{x}_2\|^2_E \geq  \tilde\delta\,\|x_2\|^2_E.
$$
Thus in particular, the operator $S^{-1} +K^*K$ is injective. Hence, \eqref{eq:9.6} is solvable, i.e., $x_1\in\dom{A_S}$ and $z_1=A_Sx_1$, if and only if 
$$
\begin{aligned}
  x_1\in L^2(\Omega),\quad \Grad x_1\in (S^{-1}+ K^*K)\,H^{\mathrm{div}}_0(\Omega),\quad \text{and}\\
  z_1=\mathrm{div} \left( \left( S^{-1}+ K^*K \right)^{-1} \mathrm{grad} \, x_1 \right);
\end{aligned}
$$
indeed then also $x_1$ and 
$$
	x_2=\left( S^{-1}+ K^*K \right)^{-1} \mathrm{grad} \, x_1\in H_0^{\mathrm{div}}(\Omega)
$$ 
satisfy $\sbm{x_1\\x_2}\in\dom{A_{ext}}$, since
$$
	\Grad x_1-K^*Kx_2=z_2=S^{-1}x_2\in S^{-1} H^{\mathrm{div}}_0(\Omega)\subset L^2(\Omega).
$$
We conclude by Theorem \ref{thm:KadmConseq} that 
\begin{equation}
  \label{eq:2}
  A_S =  \mathrm{div} \left( S^{-1}+ K^*K \right)^{-1} \mathrm{grad} 
\end{equation}
with domain
\begin{equation}
  \label{eq:3}
  \dom {A_S} = \left\{ x \in L^2(\Omega) \mid \mathrm{grad}\, x \in
    (S^{-1} + K^*K)\,H^{\mathrm{div}}_0(\Omega) \right\}
\end{equation}
generates a contraction semigroup on $L^2(\Omega)$. Here
  the multiplication by $\alpha$ in (\ref{eq:heatND}) has been replaced
  by the operator $(S^{-1}+ K^*K)^{-1}$. This makes it possible to treat
  the degenerate case, as we make explicit in the next example.

The boundary condition on the operator $A_S$ in equation
  (\ref{eq:2}) and (\ref{eq:3}) is that the normal trace of $\left( S^{-1}+ K^*K
  \right)^{-1} \mathrm{grad}\,x$ should be zero along all of the boundary, and this case is technically rather simple to deal with. To illustrate how more challenging boundary conditions (where different parts of the boundary are coupled) can be handled, we take a one-dimensional spatial domain.
  
We set $\beta(\xi) :=
\xi^{-\alpha}$, $\xi\in(0,1)$, with $\alpha \in (0,1)$. Then the corresponding multiplication operator $K=M_\beta$ maps $E_0:=\set{x\in H^1(0,1)\mid x(0)=0}$ with the $H^1(0,1)$ norm into $L^2(0,1)$, because
$$
  \left|\beta(\xi)x(\xi)\right|=\beta(\xi)\left|\int_0^\xi 1\cdot x'(\tau)\ud\tau\right|\leq \beta(\xi) \sqrt\xi \, \|x'\|_{L^2(0,1)}
  \leq \beta(\xi) \sqrt\xi \, \|x\|_{E_0}
$$
by Cauchy-Schwartz, and $\int_0^1\big(\beta(\xi)
\sqrt\xi\big)^2\ud\xi=\frac{1}{2-2\alpha}$. Hence the norm of
$M_\beta$ is bounded by $\frac{1}{\sqrt{2-2\alpha}}$, and $M_\beta^*$
is multiplication by $\overline\beta=\beta$, mapping $L^2(0,1)$
continuously into the dual $E_0'$ of $E_0$ with pivot space
$L^2(0,1)$.

Take $\kappa>0$ arbitrarily and observe that $x_1'+\beta e\in L^2(0,1)$ and $\beta\big|_{(\frac12,1)}$ bounded implies that $x_1'=(z-\beta e)\big|_{(\frac12,1)}\in L^2(\frac12,1)$. Hence $x_1\big|_{(\frac12,1)}\in H^1(\frac12,1)$ and $x_1(1)$ is well-defined. We leave it to the reader to verify that the (unbounded) adjoint of the operator
\begin{equation}
  \label{eq:1}
   A_{ext,0}= \left[ \begin{matrix} 0 & \frac{\partial}{\partial \xi}  &0 \\ \frac{\partial}{\partial \xi} &0& M_\beta^* \\ 0 &-M_\beta &0 \end{matrix} \right]
\end{equation}
with domain
$$
\begin{aligned}
  \dom{A_{ext,0}}= \left\{ \left[ \begin{matrix} x_1\\ x_2\\e \end{matrix} \right] \in L^2(0,1)^3 \right. \bigmid &\  x_2 \in H^1(0,1),~x'_1 + \beta e \in L^2(0,1), \\
     &   x_2(0)=0,~ x_1(1) =-\kappa x_2(1) \bigg\}
\end{aligned}
$$
$$
  \text{is}\qquad
   A_{ext,0}^*= -\left[ \begin{matrix} 0 & \frac{\partial}{\partial \xi}  &0 \\   \frac{\partial}{\partial \xi} &0& M_\beta^* \\ 0 &-M_\beta &0 \end{matrix} \right]\qquad\text{with domain}
$$
$$
\begin{aligned}
  \dom{A_{ext,0}^*}= \left\{ \left[ \begin{matrix} x_1\\ x_2\\e \end{matrix} \right] \in L^2(0,1)^3 \right. \bigmid &\   x_2 \in
      H^1(0,1),~ x'_1 + \beta e \in L^2(0,1),\\
     & \quad x_2(0)=0, ~ x_1(1) =\kappa x_2(1) \bigg\}.
\end{aligned}
$$
A main step in this verification is showing that $z_1\big|_{[a,1]}\in H^1(a,1)$ for all $a\in(0,1)$ whenever $(z_1,z_2,h)\in\dom{A_{ext,0}^*}$, which again follows from the boundedness of $\beta$ on every interval $[a,1]$, $a\in(0,1)$. Since both
  $A_{ext,0}$ and $A_{ext,0}^*$ are closed and dissipative, $A_{ext,0}$ is the generator of a contraction semigroup on $L^2(\Omega)^3$.

Applying Theorem \ref{thm:parahypconf} to the operator in \eqref{eq:1}
with $S=I$, we obtain
\begin{equation}
  \label{eq:4}
  A_{S,0} = \left[ \begin{matrix} 0 & \frac{\partial}{\partial \xi}  &0 \\ \frac{\partial}{\partial \xi} &0& M_\beta^* \end{matrix} \right]
  \left[ \begin{matrix} I&0 \\ 0&I \\ 0 & -M_\beta \end{matrix} \right]
  =\left[ \begin{matrix} 0 & \frac{\partial}{\partial \xi} \\ \frac{\partial}{\partial \xi} & -M_\beta^* M_\beta \end{matrix} \right],
\end{equation}
with domain
$$
\begin{aligned}
   \dom{A_{S,0}}= \left\{ \left[ \begin{matrix} x_1\\ x_2 \end{matrix} \right] \in  \left[ \begin{matrix} L^2(0,1) \\ H^1(0,1) \end{matrix} \right] \bigmid\right. & 
   x'_1 - \beta^2 x_2 \in L^2(0,1), \\ &\quad x_2(0)=0,~ x_1(1) =-\kappa x_2(1)\bigg\}.
\end{aligned}
$$
Note that the operator (\ref{eq:4}) is of a similar form as the
  operator in \eqref{eq:9}, but now the boundary conditions on $x_1$ and $x_2$ are coupled at $\xi=1$. By Theorem \ref{thm:parahypconf}, $A_{S,0}$ generates a contraction semigroup on $L^2(\Omega)^2$. 

We next apply Theorem \ref{thm:KadmConseq} to the operator $A_{ext}:=A_{S,0}$ with $S=M_{s}$, i.e., multiplication by the function $s$. The calculations here are the same as in the $n$-D case above, and the result is
\begin{equation}\label{eq:AS1beta}
\begin{aligned}
  A_{S,1}x =& \frac{\partial}{\partial \xi} \left(
    \frac{1}{s^{-1}(\xi)+\beta(\xi)^2} \frac{\partial x}{\partial \xi} \right)\quad\text{with domain} \\
  \dom{A_{S,1}}= & \bigg\{ x \in L^2(0,1) \bigmid  \frac{1}{s^{-1}+\beta^2} x' \in H^1(0,1), \\
  &\quad \left.  \left(\frac{1}{s^{-1}+\beta^2} x'\right)(0)=0,~ 
    x(1) = -\kappa \left(\frac{1}{s^{-1}+\beta^2} x'\right)(1)\right\}.
\end{aligned}
\end{equation}
Using the expression $\beta(\xi)=\xi^{-\alpha}$, this becomes
$$
\begin{aligned}
  A_{S,1}x &= \frac{\partial }{\partial \xi} \left(
   \frac{s(\xi) \xi^{2\alpha}}{1+ s(\xi)\xi^{2\alpha}}\, x'(\xi) \right)
  \qquad\text{with domain}\\
  \dom{A_{S,1}}&= \left\{ x \in L^2(0,1) \bigmid\ \frac{s(\xi) \xi^{2\alpha}}{1+ s(\xi) \xi^{2\alpha}}\,  x'(\xi) \in H^1(0,1),\right.\\
	  &\qquad\qquad\qquad \left.\left( s(\xi) \xi^{2\alpha}\, x'(\xi)\right)(0)=0,~x(1) = -\kappa\frac{s(1)}{1+s(1)} x'(1)\right\}.
\end{aligned}
$$
Here the thermal diffusivity $s(\xi) \xi^{2\alpha}\,(1+ s(\xi)\xi^{2\alpha})^{-1}$ becomes zero at $\xi=0$.

This way any thermal diffusivity that can be written as $\tilde{s}\beta^{-2}$ with $\tilde{s}$ positive, bounded and bounded away from zero can be captured. We leave it for future work to extend the situation with mixed boundary conditions to the $n$-D case.

\section{Acknowledgements}

The authors gratefully acknowledge that the anonymous referee has been most helpful with improving the manuscript.

\def\cprime{$'$} \def\cprime{$'$}
\providecommand{\bysame}{\leavevmode\hbox to3em{\hrulefill}\thinspace}
\providecommand{\MR}{}
\providecommand{\MRhref}[2]{}{}

\providecommand{\href}[2]{#2}

\appendix

\section{A lemma on unbounded adjoints}\label{app:adjoints}

The following result must be well-known in the literature, but we could not find a suitably formulated reference:

\begin{lemma}\label{lem:adjoints}
Let $H$, $K$, and $L$ be Hilbert spaces, and let $Q:K\to L$ and $R:H\to K$ be possibly unbounded operators. If $Q$ is bounded, or if $R$ is bounded and surjective, then $(QR)^*=R^*Q^*$.
\end{lemma}
\begin{proof}
The proof for the case where $Q$ is bounded is trivial. Moreover, the inclusion $R^*Q^*\subset (QR)^*$ always holds for linear operators $Q$ and $R$, as one easily shows. We finish the proof by showing that if $R$ is bounded and surjective, then the converse inclusion also holds.

Assume that there exists a $w$ such that $\Ipdp{QRx}{z}=\Ipdp{x}{w}$ for all $x\in\dom{QR}$. Then in particular $0=\Ipdp{x}{w}$ for all $x\in\Ker{R}$, so that $w\in\Ker R^\perp=\range{R^*}$, since $R^*$ has closed range by the Closed Range Theorem. Writing $w=R^*v$, we thus obtain that $\Ipdp{QRx}{z}=\Ipdp{x}{R^*v}=\Ipdp{Rx}{v}$ for all $Rx\in\dom{Q}$, again using the boundedness and surjectivity of $R$. Therefore $z\in\dom{Q^*}$ and $Q^*z=v$.

Hence $z\in\dom{(QR)^*}$ and $w=(QR)^*z$ imply $z\in\dom{R^*Q^*}$ and $w=R^*Q^*z$, i.e., that $(QR)^*\subset R^*Q^*$. 
\end{proof}

\end{document}